\documentclass{patmorin}
\usepackage{pat}
\usepackage[T1]{fontenc}
\usepackage[utf8]{inputenc}
\usepackage{mathtools}
\usepackage{thmtools}
\usepackage{thm-restate}
\usepackage[longnamesfirst,numbers,sort&compress]{natbib}
\usepackage[inline]{enumitem}

\newenvironment{clmproof}{\begin{proof}[Proof of Claim:]}{\end{proof}}

\newcommand{\jedrzej}[1]{\textcolor{ProcessBlue}{[Jędrzej: #1]}}

\newcommand{\quentin}[1]{\textcolor{Purple}{[Quentin: #1]}}
\newcommand{\gwen}[1]{\textcolor{RedOrange}{[Gwen: #1]}}

\newcommand{\suchthat}{\colon}

\DeclareMathOperator{\pw}{pw}
\DeclareMathOperator{\lpw}{lpw}

\DeclareMathOperator{\td}{td}
\DeclareMathOperator{\ltd}{ltd}

\DeclareMathOperator{\diam}{diam}
\DeclareMathOperator{\radius}{rad}

\DeclareMathOperator{\dist}{dist}

\newcommand{\NN}{\mathbb{N}}

\let\le\leqslant
\let\ge\geqslant
\let\leq\leqslant
\let\geq\geqslant

\let\subset\subseteq

\usepackage{comment}

\title{\MakeUppercase{Excluding an apex-forest or a fan as quickly as possible}}

\author{Quentin Claus\thanks{D\'epartement de Mathématiques, Universit\'e libre de Bruxelles, Belgium ({\tt quentin.claus@ulb.be}). Q.\ Claus is supported by the Belgian National Fund for Scientific Research (FNRS).} \qquad Jędrzej Hodor\thanks{Theoretical Computer Science Department, 
Faculty of Mathematics and Computer Science and Doctoral School of Exact and Natural Sciences, Jagiellonian University, Krak\'ow, Poland (\texttt{jedrzej.hodor@gmail.com}). J.\ Hodor is supported by a Polish Ministry of Education and Science grant (Perły Nauki; PN/01/0265/2022).} \qquad Gwenaël Joret\thanks{D\'epartement d'Informatique, Universit\'e libre de Bruxelles, Belgium ({\tt gwenael.joret@ulb.be}). G.\ Joret is supported by the Belgian National Fund for Scientific Research (FNRS).} \qquad Pat Morin\thanks{School of Computer Science, Carleton University, Ottawa, Canada (\texttt{morin@scs.carleton.ca}). Research supported by NSERC.}}

\date{}

\begin{document}

\maketitle

\begin{abstract}
	We show that every graph $G$ excluding an apex-forest $H$ as a minor has layered pathwidth at most $|V(H)|-2$, and that every graph $G$ excluding an apex-linear forest (such as a fan) $H$ as a minor has layered treedepth at most $|V(H)|-2$.
	We further show that both bounds are optimal.
	These results improve on recent results of Hodor, La, Micek, and Rambaud (2025):
	The first result improves the previous best-known bound by a multiplicative factor of $2$, while the second strengthens a previous quadratic bound. 
	In addition, we reduce from quadratic to linear the bound on the $S$-focused treedepth $\td(G,S)$ for graphs $G$ with a prescribed set of vertices $S$ excluding models of paths in which every branch set intersects~$S$. 	
\end{abstract}

\section{Introduction}

A central theme in the celebrated proof of Robertson and Seymour's Graph Minor Theorem is the interplay between graph decompositions and excluded minors.  For example, excluding a planar graph (as a minor) is equivalent to having bounded treewidth~\cite{GM5,Chuzhoy2021};
excluding a forest is equivalent to having bounded pathwidth~\cite{GM1,BIENSTOCK1991274,Diestel1995}; and
excluding a path is equivalent to having bounded treedepth.

Recently, \emph{layered} graph decompositions have been used to resolve a number of longstanding open problems on planar graphs, other minor-closed graph classes, and some non-minor closed graph classes~\cite{BBEGLPS24, DEMWW22, DF18, DMMY25, DMY21, layered-treewidth, layered-pathwidth, wood.liu:clustered, quickly-excluding-apex-forest, SW20}.  In these applications, the parameter associated with the layered decomposition usually has a significant impact on the bound in the application.  For example, \citet{DEMWW22} show that bounded-degree graphs of layered treewidth at most $t$ admit $3$-colorings with clustering $O(t^{3})$. 

In several cases, the existence of such layered graph decompositions is implied by an excluded minor, where the excluded graph is an apex-type graph.\footnote{For a graph family $\mathcal{F}$, a graph $G$ is $\mathcal{F}$-apex if $G$ contains a vertex $v$ such that $G-v\in\mathcal{F}$. For example, a graph is an \defin{apex-planar graph} if it can be made planar by the removal of at most one vertex; a graph is an \defin{apex-forest} if it can be made acyclic by the removal of at most one vertex; and a graph is an \defin{apex-linear forest} if it can be made a subgraph of a path by the removal of at most one vertex. For historical reasons an apex-planar graph is simply referred to as an \defin{apex graph}.}  For example, excluding an apex graph is equivalent to having bounded layered treewidth~\cite{layered-treewidth};
excluding an apex-forest is equivalent to having bounded layered pathwidth~\cite{layered-pathwidth,quickly-excluding-apex-forest}; and
excluding a fan is equivalent to having bounded layered treedepth~\cite{quickly-excluding-apex-forest}.
We revisit these last two results, obtaining tighter bounds in the equivalences.
Definitions of layered pathwidth $\lpw(\cdot)$ and layered treedepth $\ltd(\cdot)$ are given in \cref{sec:definitions}.

\citet*{layered-pathwidth} prove that there exists a function $f$ such that for every apex-forest $H$,
every $H$-minor free graph $G$ satisfies $\lpw(G) \leq f(|V(H)|)$.
\citet*{quickly-excluding-apex-forest} give a much shorter proof and an explicit formula for~$f$, namely $f(k) = 2k-3$ for $k\geq 2$.
We find the best possible function: $f(k)=k-2$.
\begin{thm} \label{thm:lpw}
	For every apex-forest $H$ with at least three vertices,
	and for every $H$-minor-free graph $G$,
	$\lpw(G)\leq |V(H)|-2$.
\end{thm}
\citet{quickly-excluding-apex-forest} prove that, for every apex-linear forest $H$, every $H$-minor-free graph $G$ satisfies $\ltd(G)\leq \binom{|V(H)|-1}{2}$.  Again, we find the best possible bound.
\begin{thm} \label{thm:ltd}
	For every apex-linear forest $H$ with at least three vertices, and for every $H$-minor-free graph~$G$,
	$\ltd(G)\leq |V(H)|-2$.
\end{thm}
In~\cref{sec:lower-bounds}, we show that the bounds in~\Cref{thm:lpw,thm:ltd} are optimal.

\citet{quickly-excluding-apex-forest} obtained the results by studying the interplay between forbidding graphs as minors and graph decompositions in the so-called \emph{focused} setting (note that it can be seen as a special case of a much more general theory of colorful graphs introduced by Protopapas, Thilikos, and Wiederrecht~\cite{colorful}).
In this setting, we consider graphs with a prescribed subset of vertices that we see as \emph{important}.
We study graphs excluding minor-models such that every branch set contains an important vertex (\emph{rooted models}).
The corresponding graph decompositions are harder to describe.
Intuitively, only a part of the graph is decomposed, and this part should capture the complexity of the important vertices.
For precise definitions of focused pathwidth and focused treedepth see~\Cref{sec:focused-decompositions}.
Statements analogous to the results above on excluded minors also hold in this setting \cite{quickly-excluding-apex-forest}:
Excluding a planar graph as an outer-rooted minor (see~\citet{quickly-excluding-apex-forest} for the definition) is equivalent to having bounded focused treewidth;
excluding a forest as a rooted minor is equivalent to having bounded focused pathwidth; and
excluding a path as a rooted minor is equivalent to having bounded focused treedepth.
Let us state the last result more precisely.
For every graph $G$ and every $S\subseteq V(G)$ (set of important vertices), if $G$ does not contain an $S$-rooted model of the $\ell$-vertex path $P_\ell$, then the $S$-focused treedepth $\td(G,S)$ of $G$ is at most $\binom{\ell}{2}$ \cite[Theorem~6]{quickly-excluding-apex-forest}.
\citet{quickly-excluding-apex-forest} ask if this quadratic bound can be improved to linear. We answer this question in the affirmative:

\begin{thm} \label{thm:main}
	For every positive integer $\ell$, for every graph $G$, and for every
	set $S\subseteq V(G)$, if $G$ has no $S$-rooted model of $P_\ell$, then $\td(G,S) \le 2\ell-2$.
\end{thm}


\subsection{Consequences}\label{sec:consequences}

\citet{quickly-excluding-apex-forest} present consequences of analogous results to~\Cref{thm:lpw,thm:ltd} and remark that any improvement to either result (such as in this paper) yields an improvement to the corresponding consequence.
For completeness, we state these.

Recall that the \defin{diameter} of a connected graph $G$, denoted by \defin{$\diam(G)$}, is the maximal distance between two vertices in $G$ taken over all pairs of vertices in $G$.
Since paths can have arbitrarily large treedepth, connected graphs of bounded treedepth have bounded diameter.
In general, the converse is not true, as e.g.\ every fan has diameter at most $2$ and fans have arbitrarily large treedepth.\footnote{A \defin{fan} is a graph obtained from a path by adding a dominant vertex.}
Thus, a natural question to ask is for which minor-closed graph classes the converse is true.
As shown in~\cite[Corollary~3]{quickly-excluding-apex-forest}, these are minor-closed graph classes excluding a fan (or more generally an apex-linear forest).
Our techniques yield an improvement in the multiplicative constant.
Additionally, we note that the diameter in the bound can be replaced with the radius.
Moreover, our bound is tight (see~\Cref{lem:lb}).
Recall that the \defin{radius} of a connected graph $G$, denoted by \defin{$\radius(G)$}, is a minimum integer $r$ such that there exists a vertex $v$ in $G$ in distance at most $r$ to every vertex in $G$.



\begin{cor}\label{cor:td-diam}
	For every apex-linear forest $H$ with at least two vertices, and for every connected graph $G$, if $G$ is $H$-minor-free, then $\td(G) \leq (|V(H)|-2)\cdot \radius(G)+1$.
\end{cor}

We obtain a similar result for pathwidth.

\begin{cor}\label{cor:pw-diam}
	For every apex-forest $H$ with at least three vertices, and for every connected graph $G$, if $G$ is $H$-minor-free, then $\pw(G) \leq (|V(H)|-2)\cdot \radius(G)$.
\end{cor}

The remainder of the paper is organized as follows:
In \cref{sec:definitions} we give definitions of layered treedepth, layered pathwidth, focused treedepth, and focused pathwidth.
In~\Cref{sec:radius}, we discuss~\Cref{cor:td-diam} and~\Cref{cor:pw-diam}.
In \cref{sec:td-gs} we prove \cref{thm:main}.
In \cref{sec:ltd} we prove \cref{thm:ltd}.
In \cref{sec:lpw} we prove \cref{thm:lpw}.
In \cref{sec:lower-bounds} we give lower bounds constructions.

\section{Definitions} \label{sec:definitions}

For a graph $G$, we denote by \defin{$V(G)$} its set of vertices and by \defin{$E(G)$} its set of edges.
For a vertex $v$ in a graph $G$, the \defin{neighbourhood} of $v$ in $G$ is $\mathdefin{N_G(v)}=\{w\in V(G):vw\in E(G)\}$.
The neighbourhood of a subset of vertices $X$ of a graph $G$ is $\mathdefin{N_G(X)}=\bigcup_{v\in X}N_G(v)\setminus X$.

We treat paths as directed objects: Given a graph $G$, a \defin{path} in $G$ is a sequence of distinct vertices $v_1,\ldots,v_k$ with $v_iv_{i+1}\in E(G)$ for each $i\in\{1,\ldots,k-1\}$. For a path $P:=v_1,\ldots,v_k$, the \defin{first vertex} $v_1$ and the \defin{last vertex} $v_k$ of $P$ are the \defin{endpoints} of $P$ and $v_2,\ldots,v_{k-1}$ are the \defin{internal vertices} of $P$. (The vertex of a $1$-vertex path is both the first and last vertex of the path.)
Given a graph $G$ and two sets $A,B \subseteq V(G)$, an \defin{$A$--$B$ path} in $G$ is a path $P=v_1,\ldots,v_k$ with $V(P)\cap A = \{v_1\}$ and $V(P)\cap B = \{v_k\}$.

Given a graph $G$, and subsets of vertices $A$, $B$, and $X$ of $G$, we say that $X$ \defin{separates} $A$ and $B$ in $G$ if every connected component of $G - X$ contains vertices of at most one among $A$ and $B$.

A \defin{rooted tree} is a tree $T$ with a designated root vertex $v\in V(T)$.  For a vertex $x$ in a rooted tree $T$,  we use the notation \defin{$P_T(x)$} to denote the unique path from $x$ to the root of $T$.
Each path of this form is called \defin{vertical} in $T$.
A vertex $y$ of a rooted tree $T$ is a \defin{$T$-ancestor} of a vertex $x$ of $T$ if and only if $y\in V(P_T(x))$.
If $y$ is a $T$-ancestor of $x$, then $x$ is a \defin{$T$-descendant} of $y$.  (Note that $x$ is an ancestor and a descendant of itself.)  For any vertex $x$ in a rooted tree $T$, $\mathdefin{T_x}$ denotes the subtree of $T$ containing $x$ and all descendants of $x$.
The \defin{lowest common ancestor} of a non-empty subset $S\subseteq V(T)$ is the vertex $x_0$ of $T$ that maximizes $|V(P_T(x_0))|$ among all vertices in $\{x\in V(T)\suchthat S\subseteq V(T_x)\}$.
The \defin{vertex-height} of a rooted tree $T$ is the maximum number of vertices $\max\{|V(P_T(v))|:v\in V(T)\}$ on a path from the root to a leaf in~$T$.

A \defin{rooted forest} is the union of pairwise vertex-disjoint rooted trees.  A vertex $y$ in a rooted forest $F$ is an \defin{$F$-ancestor} of a vertex $x$ (and $x$ is an \defin{$F$-descendant} of $y$)  if $x$ and $y$ belong to the same tree $T$ in $F$ and $y$ is a $T$-ancestor of $x$. When the forest $F$ (or tree $T$) is clear from context we will simply say that $y$ is an \defin{ancestor} of $x$ (and $x$ is a \defin{descendant} of $y$).  The \defin{vertex-height} of a rooted forest is the maximum vertex-height over all trees in the forest.  A \defin{linear forest} is a forest, all of whose connected components are paths.

Given a connected graph $G$ and a vertex $r\in V(G)$, an \defin{$r$-rooted DFS tree of $G$} is any rooted spanning tree of $G$ that is rooted at $r$ and has the property that $N_G(V(T_x)) \subseteq V(P_T(x))$, for each $x\in V(T)$. In words, all edges of $G$ leaving $T_x$ lead to an ancestor of $x$.
It is well-known that such a spanning tree always exists.

\subsection{Layered treedepth and pathwidth}\label{sec:parameters}

Let $\N$ be the set of non-negative integers.
A \defin{layering} of a graph $G$ is a sequence $(L_i \suchthat i \in \N)$ of pairwise disjoint subsets of $V(G)$ whose union is $V(G)$ and such that for every edge $uv \in E(G)$ there is a non-negative integer $i$ such that $u,v \in L_i \cup L_{i+1}$.
The \defin{closure} of a rooted forest $F$ is the graph obtained from $F$ by adding edges between every pair of distinct vertices where one vertex of the pair is an ancestor of the other.

We say that $F$ is an \defin{elimination forest} of a graph $G$ if $V(F)=V(G)$ and $G$ is a subgraph of the closure of $F$.
The \defin{treedepth} of a $G$, denoted by \defin{$\td(G)$}, is $0$ if $G$ is empty, and otherwise is the minimum vertex-height of an elimination forest of $G$.
When $\mathcal{L} = (L_i \suchthat i \in \N)$ is a layering of a graph $G$ and $F$ is an elimination forest of $G$, then the \defin{width of $(F,\mathcal{L})$} is the maximum value of $|L_i \cap R|$ where $i \in \N$ and $R = V(P_F(x))$ for some $x \in V(G)$.
The \defin{layered treedepth} of $G$ is the minimum width over such pairs $(F,\mathcal{L})$.

A \defin{path decomposition} of a graph $G$ is a sequence $\mathcal{W} = (W_1,\dots,W_m)$ of subsets of $V(G)$ (called \defin{bags} of $\mathcal{W}$) such that for each edge $uv\in E(G)$ there is $i \in \{1,\dots,m\}$ with $u,v \in W_i$ and for each vertex $v\in V(G)$ the set of $i \in \{1,\dots,m\}$ with $v \in W_i$ is a non-empty interval.
The \defin{width} of $\mathcal{W}$ is the maximum size of $W_i$ over $i \in \{1,\dots,m\}$ minus $1$.
The \defin{pathwidth} of $G$, denoted by \defin{$\pw(G)$}, is the minimum width of a path decomposition of $G$.
When $\mathcal{L} = (L_i \suchthat i \in \N)$ is a layering of a graph $G$ and $\mathcal{W} = (W_1,\dots,W_m)$ is a path decomposition of $G$, then the \defin{width of $(\mathcal{W},\mathcal{L})$} is the maximum value of $|L_i \cap W_j|$ where $i \in \N$ and $j \in \{1,\dots,m\}$.
The \defin{layered pathwidth} of $G$ is the minimum width over such pairs $(\mathcal{W},\mathcal{L})$.
It is easy to check that $\lpw(G) \leq \ltd(G)$ for every graph $G$.

\subsection{Focused treedepth and pathwidth} \label{sec:focused-decompositions}

For graphs $G$ and $H$, a \defin{model of $H$} in $G$ is a set $\mathcal B:=\{B_x\suchthat x\in V(H)\}$ of pairwise vertex-disjoint connected subgraphs of $G$ indexed by the vertices of $H$, called \defin{branch sets}, such that, for each edge $xy$ in $H$, $G$ contains an edge with one endpoint in $V(B_x)$ and one endpoint in $V(B_y)$.
Additionally, for brevity, we define $V(\mathcal B):= \bigcup_{x\in V(H)}V(B_x)$.
We say that $G$ contains $H$ as a \defin{minor} if there is a model of $H$ in $G$; otherwise, we say that $G$ is \defin{$H$-minor-free}.
Given $S\subseteq V(G)$, an $H$-model $\{B_x\suchthat x\in V(H)\}$ in $G$ is \defin{$S$-rooted} if $B_x$ contains at least one vertex in $S$, for each $x\in V(H)$.

An example of a focused graph-parameter is the focused treedepth.
For a graph $G$, let \defin{$\mathcal{C}(G)$} denote the set of connected components of $G$.
For a graph $G$ and a set $S\subseteq V(G)$, the \defin{$S$-focused treedepth of $G$}, also called simply the \defin{treedepth of $(G,S)$}, is defined recursively as
\[
	\mathdefin{\td(G,S)} :=
	\begin{cases}
		0                                                        & \text{if $S=\emptyset$,}                             \\
		\max\{\td(C, S \cap V(C))\suchthat C\in\mathcal{C}(G)\}  & \text{if $S\neq\emptyset$ and $|\mathcal{C}(G)|>1$,} \\
		1+\min\{\td(G-v, S \setminus \{v\})\suchthat v\in V(G)\} & \text{otherwise.}
	\end{cases}
\]
It is an easy exercise to show that for every graph $G$, we have $\td(G,V(G)) = \td(G)$.

Sometimes, it is handy to use an equivalent definition of focused treedepth.
Given a graph $G$ and $S\subseteq V(G)$, an \defin{elimination forest of $(G, S)$} is a rooted forest $F$ such that $S\subset V(F)\subset V(G)$, $F$ is an elimination forest of $G[V(F)]$, and for every connected component $C$ of $G-V(F)$, either $N_G(V(C)) = \emptyset$ or there exists a root to leaf path in $F$ whose vertex set contains $N_G(V(C))$.
It is easy to check, see~\cite[Section~5]{quickly-excluding-apex-forest}, that $\td(G, S)$ is equal to the minimum vertex-height of an elimination forest of $(G, S)$.
The following observation is straightforward and will be used implicitly.

\begin{obs}
    Let $G$ be a graph and $G'$ be a subgraph of $G$.
    Let $S \subset V(G)$ and $S' \subset S \cap V(G')$.
    Then, $\td(G',S') \leq \td(G,S)$.
\end{obs}

A second example of a focused graph-parameter is the focused pathwidth.
Let $G$ be a graph and $S \subseteq V(G)$.
A \defin{path decomposition of $(G,S)$} consists of an induced subgraph $J$ of $G$ such that $S$ is contained in $V(J)$, and a path decomposition $\mathcal{W}$ of $J$ such that  for every connected component $C$ of $G-V(J)$, there exists a bag of $\mathcal{W}$ containing all the neighbours of $V(C)$ in $G$.
The \defin{$S$-focused pathwidth of $G$}, also called simply the \defin{pathwidth of $(G,S)$}, denoted by \defin{$\pw(G,S)$} is the minimum width of a path decomposition of~$(G,S)$.
Again, for every graph $G$, we have $\pw(G,V(G)) = \pw(G)$.
Note that \citet[Theorem~7]{quickly-excluding-apex-forest} proved that for every positive integer $\ell$, for every forest $F$ on $\ell$ vertices, for every graph $G$, and for every set $S \subset V(G)$, if $G$ has no $S$-rooted model of $F$, then $\pw(G,S) \leq 2\ell-2$.

\section{Excluding a fan in bounded-radius graphs} \label{sec:radius}

To obtain~\Cref{cor:pw-diam} and~\Cref{cor:td-diam}, we need slightly stronger variants of~\Cref{thm:lpw} and~\Cref{thm:ltd}, which will fortunately follow directly from our proofs.

\begin{thm} \label{thm:lpw-stronger}
	For every apex-forest $H$ with at least three vertices, for every $H$-minor-free graph $G$, and for every vertex $v\in V(G)$, there exists a path decomposition $\mathcal{W}$ of $G$, and a layering $\mathcal{L}=(L_i\colon i\in \mathbb{N})$ of $G$, such that $L_0=\{v\}$, and the width of $(\mathcal W, \mathcal{L})$ is at most  $|V(H)|-2$.
\end{thm}

\Cref{thm:lpw-stronger} implies \cref{cor:pw-diam} as follows: By choosing $v$ so that $\max\{\dist_G(v,w):w\in V(G)\}=\radius(G)$, the layering $\mathcal{L}=(L_i:i\in\N)$  obtained from \cref{thm:lpw-stronger} has at most $\radius(G)+1$ non-empty layers, including $L_0$. Then, for every bag $W\in\mathcal{W}$,   \[|W|=\sum_{i=0}^\infty|W\cap L_i|\le 1+\max\{|W\cap L_i|\suchthat i\geq 1\} \cdot \radius(G)\le 1+(|V(H)|-2)\cdot \radius(G).\]
Thus, $\max\{|W|\suchthat W\in\mathcal{W}\}\le 1+(|V(H)|-2)\cdot \radius(G)$, so $\mathcal{W}$ is a path decomposition of $G$ of width at most $(|V(H)|-2)\cdot\radius(G)$.

\begin{thm} \label{thm:ltd-stronger}
	For every apex-linear forest $H$ with at least three vertices, for every $H$-minor-free graph $G$, and for every vertex $v\in V(G)$, there exists an elimination forest $F$ of $G$, and a layering $\mathcal{L}=(L_i \colon i\in \mathbb N)$ of $G$, such that $L_0=\{v\}$, and the width of $(F, \mathcal{L})$ is at most  $|V(H)|-2$.
\end{thm}

\Cref{thm:ltd-stronger} is a strengthening of \cref{thm:ltd} that implies \cref{cor:td-diam} by the same argument used in the previous paragraph.

\section{Focused treedepth}\label{sec:td-gs}

In this section, we prove \cref{thm:main}.
Let $G$ be a connected graph, let $S \subseteq V(G)$, and let $T$ be a DFS tree of $G$.  For a positive integer $\ell$, a pair $J:=(P,\{Q_1,\ldots,Q_{\ell}\})$ is called a \defin{$(G,S,T)$-path} of \defin{order} $\ell$ if
\begin{enumerate}[nosep,nolistsep,label=(J\arabic*)]
	\item\label{vertical_spine} $P:=P_T(u)$ for some vertex $u$ of~$G$. The path $P$ is called the \defin{spine} of $J$ and $u$ is called the \defin{coccyx} of $J$.
	\item\label{disjoint_ribs} $Q_1,\ldots,Q_{\ell}$ are pairwise vertex-disjoint $V(P)$--$S$ paths in $G$. For each $i\in\{1,\ldots,\ell\}$, $Q_i$ is a \defin{rib} of $J$, and the first vertex of $Q_i$ (the vertex in $P$) is an \defin{attachment} in $J$.
	\item\label{coccyx_attachment} The coccyx $u$ of $J$ is an attachment in $J$.
\end{enumerate}

The proof of the following observation is immediate.
(A detail that is perhaps worth pointing out: It follows from the definition of $A$--$B$ paths that each rib of the $(G,S,T)$-path has an attachment on the spine and is otherwise disjoint from the spine.)

\begin{obs}\label{gst_pl}
	Let $\ell$ be a positive integer.
	Let $G$ be a connected graph, let $T$ be a DFS tree of $G$, and let $S\subseteq V(G)$.  If $G$ has a $(G,S,T)$-path of order $\ell$ then $G$ contains an $S$-rooted model of~$P_\ell$.
\end{obs}

The heart of the proof of \Cref{thm:main} is the following technical lemma.

\begin{lem} \label{lem:main}
	Let $\ell$ be a positive integer.
	Let $G$ be a connected graph, let $T$ be a DFS tree of $G$, let $S\subseteq V(G)$, and let $X\subseteq V(G)\setminus S$ be such that
	\begin{enumerate}[nosep,nolistsep,label=\rm(\alph*),ref={Assumption (\alph*)}]
		\item $x$ is a $T$-ancestor of $v$ for each $x\in X$ and $v\in S$; \label{item:X-over-S}
		\item $\td(G-X, S)\ge 2\ell-1$. \label{item:td-G-X-Y}
	\end{enumerate}
	Then $G$ contains a $(G,S,T)$-path $J$ of order $\ell$ such that
	\begin{enumerate}[nosep,nolistsep,label=\rm(\roman*),ref={Requirement (\roman*)}]
		\item every vertex of $X$ lies in the spine of $J$; \label{item:X-spine}\label{item:first}
		\item no vertex of $X$ is an attachment in $J$. \label{item:X-degree} \label{item:last}
	\end{enumerate}
\end{lem}

Before proving \cref{lem:main}, we show that it implies \cref{thm:main}.  Let $G$ be a graph and let $S\subseteq V(G)$ be such that $G$ contains no $S$-rooted model of $P_\ell$.
Let $\ell':=\lfloor(\td(G,S)+1)/2\rfloor$, so that $2\ell'-1\le \td(G,S)\le 2\ell'$.
Observe that \cref{thm:main} holds trivially if $\td(G,S) = 0$, so we may assume $\td(G,S)\geq1$, and thus $\ell'\geq 1$.
Let $C$ be a connected component of $G$ such that $\td(C,S')=\td(G,S)$, where $S':=S\cap V(C)$.
Let $T$ be a DFS tree of $C$ and apply \cref{lem:main} with $C$, $S'$, $T$, $X=\emptyset$, and $\ell'$ to obtain a $(C,S',T)$-path $J$ in $C$ of order $\ell'$.
By \cref{gst_pl}, $C$ contains an $S'$-rooted model of $P_{\ell'}$.  Since $G$ contains no $S$-rooted model of $P_{\ell}$, $\ell' < \ell$, so $\td(G,S) \le 2\ell' \leq 2(\ell-1) = 2\ell-2$, as desired.

We now proceed with the proof of \cref{lem:main}, see~\Cref{fig:outline} for an intuitive outline of the proof.

\begin{figure}[tp]
	\centering
	\includegraphics{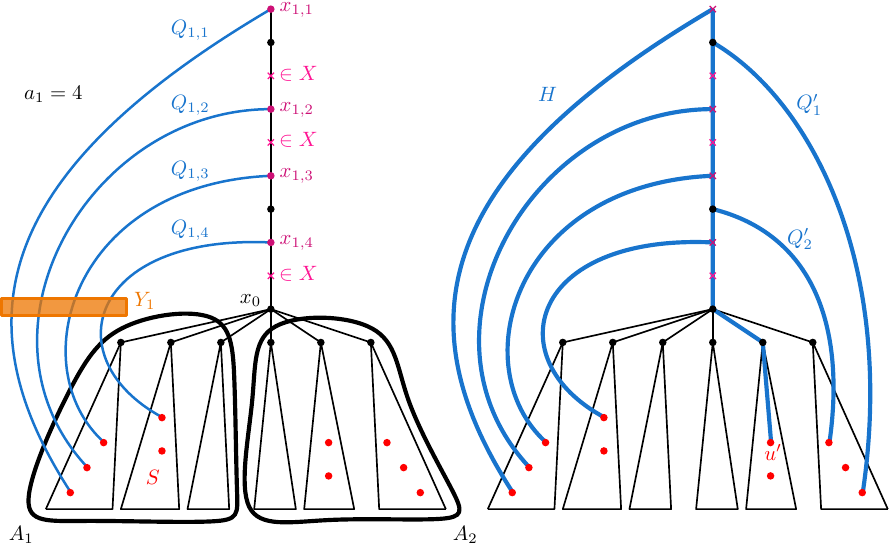}
	\caption{
	An outline of the proof of~\Cref{lem:main}.
	Note that the notation is consistent with the actual proof that we give.
	The set $X$ represents vertices that were selected as attachments in $J$ at higher levels of induction and must therefore not be used as attachments in lower levels of induction (\ref{item:X-degree}).
	Suppose that we work with $\ell = 7$.
	The vertex $x_0$ is the lowest common ancestor of $S$ in $T$.
	In the general case, we split the descendants of $x_0$ into two groups: $A_1$ and $A_2$ both containing elements of $S$.
	Suppose that (as in the figure), there are $a_1 = 4$ pairwise vertex-disjoint paths connecting $S \cap A_1$ and $V(P_T(x_0)) \setminus X$.
	Then, we can add $\{x_{1,1},x_{1,2},x_{1,3},x_{1,4}\}$ to $X$ to obtain $X_2'$ and call induction for $X_2'$, $S_2' = S \cap A_2$, and $\ell-a_1$.
	To this end, we need to assume that $\td(G-X_2',S_2') \geq 2(\ell-a_1)-1$.
	If this holds, then by induction, we obtain a $(G,S_2',T)$-path of order $\ell-a_1 = 3$ whose attachments avoid $X_2'$.
	On the right side of the figure, we show how we combine all these to obtain a $(G,S,T)$-path of order $\ell = 7$.
	When $\td(G-X_2',S_2') < 2(\ell-a_1)-1$ we can try to argue symmetrically.
	If this fails, it means that $\td(G-X_1',S_1') < 2(\ell-a_2)-1$.
	Having both, we argue that $\td(G-X,S) < 2\ell-1$, which will be a contradiction.
	The argument is intuitively simple.
	We assume $a_1 \leq a_2$ and we eliminate vertices of $X_1 \cup Y_1$ where $Y_1$ is a set of $a_1$ elements separating $A_1 \cap S$ from $V(R) \setminus X$.
	It suffices to argue that for every connected component $C$ of $G - (X \cup X_1 \cup Y_1)$, we have $\td(C, S \cap V(C)) < 2\ell - 2a_1-1$.
	}
	\label{fig:outline}
\end{figure}

\begin{proof}[Proof of \cref{lem:main}.]
	The proof is by induction on $\ell$.
	Throughout this proof, any mention of ancestor/descendant/child relations is with respect to the DFS tree $T$.

	If $\ell=1$ then, since $\td(G-X,S)\ge 1$, $S$ is non-empty.  Let $v$ be a vertex of $S$ that minimizes $|P_T(v)|$, let $P:=P_{T}(v)$, and let $Q_1$ be the $1$-vertex path that contains $v$.  Then $J:=(P,\{Q_1\})$ is a $(G,S,T)$-path of order $1$.
	By \ref{item:X-over-S}, the spine of $J$ contains every vertex in $X$, so $J$ satisfies \ref{item:X-spine}.
	Since $v\in S$ is the only attachment of $J$ and $X\subseteq V(G)\setminus S$, $J$ satisfies \ref{item:X-degree}.  Thus $J$ is a $(G,S,T)$-path of order $1$ satisfying \ref{item:first} and \ref{item:last}.  We now assume that $\ell\ge 2$.

	By the definition of $\td(\cdot,\cdot)$, there exists a connected component $C$ of $G-X$ such that $\td(C,S\cap V(C))=\td(G-X,S)\ge 2\ell-1$. Therefore, $|S\cap V(C)|\ge 2\ell-1 >0$. Let $x_0$ be the lowest common ancestor of $S\cap V(C)$ in $T$.

	First, suppose that $x_0\in S$.
	Let $X' = X \cup \{x_0\}$ and $S' = (S\cap V(C)) \setminus \{x_0\}$.
	We claim that \ref{item:X-over-S} and \ref{item:td-G-X-Y} are satisfied for $X'$, $S'$, and $\ell' := \ell-1$.
	By the definition of $x_0$ and \ref{item:X-over-S}, $x$ is an ancestor of $v$ for each $x \in X'$ and $v \in S'$, and by~\ref{item:td-G-X-Y},
	\[\td(G-X',S') \ge \td(C-x_0,(S\cap V(C))\setminus\{x_0\}) \geq \td(C,S\cap V(C))-1 \geq 2\ell-2 > 2\ell'-1.\]
	Therefore, we can apply induction to  $X'$, $S'$, and $\ell'=\ell-1$ obtaining a $(G,S',T)$-path $J':=(P, \{Q_1,\ldots,Q_{\ell-1}\})$ in $G$ of order $\ell-1$ satisfying~\ref{item:first} and \ref{item:last} for $X'$.
	Since $x_0 \in X'$, $x_0$ is not an attachment in $J'$.
	Recall that $S \subset V(G) \setminus X$, hence, $x_0 \notin X$.
	Taking $Q_\ell$ to be the one-vertex path $x_0$, it follows that $J:=(P,\{Q_1,\ldots,Q_{\ell}\})$ is a $(G,S,T)$-path in $G$ of order $\ell$ satisfying~\ref{item:first} and \ref{item:last} for $X$.

	Next, suppose that $x_0 \notin S$ and let $R = P_T(x_0)$.
	Note that $X \subseteq V(R)$ by~\ref{item:X-over-S}. Let $Z$ be the set of all children $x$ of $x_0$ in $T$ such that $V(T_x)\subseteq V(C)$. Let $Z_0 \subseteq Z$ be the set of all the children $x$ of $x_0$ such that $S\cap V(C)$  intersects $V(T_x)$.  Since $x_0 \notin S$, it follows that $|Z_0| > 1$.
	Let $Z_1$ and $Z_2$ be any partition of $Z$ such that $Z_i \cap Z_0 \neq \emptyset$ for each $i \in \{1,2\}$.  For each $i\in\{1,2\}$, let $A_i = \bigcup_{x \in Z_i} V(T_{x})$ and let $G_i = G[(A_i \cup V(R)) \setminus X ]$.  Since $T$ is a DFS tree, $V(R)$ separates $A_1$ and $A_2$ in $G$.

Let $\{1,2\} = \{i,j\}$.
We claim that $G_i$ contains a $(V(R)\setminus X)$--$(S\cap A_i)$ path.  Indeed, since $A_j\cup A_i\subseteq V(C)$ and $V(C)\subseteq V(G-X)$, the graph $G-X$ contains an $A_j$--$(S\cap A_i)$ path $Q$. Since $V(R)$ separates $A_i$ and $A_j$ in $G$, this path necessarily contains a vertex in $V(R)\setminus X$. Let $x$ be the last vertex of $Q$ in $V(R)\setminus X$. Then the subpath of $Q$ that begins at $x$ and leads to the last vertex of $Q$ is a $(V(R)\setminus X)$--$(S\cap A_i)$ path in $G_i$.
	Let $a_i$ be the maximum number of pairwise vertex-disjoint $(V(R) \setminus X)$--$(S \cap A_i)$ paths in $G_i$ (we have just proved that $a_i \geq 1$) and fix such paths $Q_{i,1},\dots,Q_{i,a_i}$.
    Consider the case where $a_i \geq \ell$.
    Let $u$ be the vertex of $A_i \cap (V(Q_{i,1}) \cup \dots \cup V(Q_{i,\ell}))$ of minimum depth.
    Without loss of generality, assume that $u \in V(Q_{i,1})$. 
    Let $Q_{i,1}'$ be the subpath of $Q_{i,1}$ that starts in $u$ and ends in $S \cap A_i$.
    Then, $(P_T(u), \{Q_{i,1}',Q_{i,2},\dots,Q_{i,\ell}\})$ is a $(G, S, T)$-path of order $\ell$ satisfying \ref{item:first} and \ref{item:last}, and we are done.  Thus, we may assume that $a_i < \ell$.
	By Menger's Theorem, there exists $Y_i \subseteq V(G_i)$ such that $|Y_i| = a_i$ and there is no $(V(R) \setminus X)$--$(S \cap A_i)$ path in $G_i - Y_i$.  For each $s\in\{1,\ldots,a_i\}$, let $x_{i,s}$ be the first vertex of $Q_{i,s}$.  Let $X_i:=\{x_{i,1},\ldots,x_{i, a_i}\}$, and note that the sets $X_i$ and $Y_i$ are not necessarily disjoint.

	Let $X'_j := X \cup X_i$ and let $S'_j := S \cap A_j$.  We will now argue that $\td(G - X'_j, S'_j) \geq 2(\ell-a_i)-1$ implies the existence of a $(G,S,T)$-path $J$ of order $\ell$ that satisfies \ref{item:first} and \ref{item:last}. Assume that $\td(G - X'_j, S'_j) \geq 2(\ell-a_i)-1$.  By \ref{item:X-over-S} and the definition of $X_i$, $x$ is an ancestor of $v$ for each $x \in X'_j$ and $v \in S'_j$, so $X'_j$ and $S'_j$ satisfy \ref{item:X-over-S}.  By assumption, $X'_j$ and $S'_j$ satisfy \ref{item:td-G-X-Y} for $\ell':=\ell-a_i>0$.  Therefore, we can apply induction to $X'_j$, $S'_j$, and $\ell'$,
	to obtain a $(G,S'_j,T)$-path $J':=(P',\{Q'_1,\ldots,Q'_{\ell-a_i}\})$ in $G$ of order $\ell-a_i$ satisfying~\ref{item:first} and \ref{item:last} for $X'_j$.

	By \ref{vertical_spine}, there exists $u'\in V(T)$ such that $P'=P_T(u')$.
	We claim that $u'\in V(R) \cup A_j$.
	For the sake of contradiction, assume otherwise.  Then, since $T$ is a DFS spanning tree, $V(P'-u')$ separates $V(T_{u'})$ from $A_j$ in $G$.  By \ref{coccyx_attachment}, one of the ribs of $J'$ is a $V(P')$--$S'_j$ path $Q'_t$ that begins at $u'$.  Since $S'_j\subseteq A_j$, $Q'_t-u'$ contains a vertex $z$ of $P'-u'$.
	Thus, $\{u', z\} \subseteq V(Q'_t)\cap V(P')$. However, since $Q'_t$ is a $V(P')$--$S'_j$ path, we must have $V(Q'_t)\cap V(P')=\{u'\}$, which is a contradiction.
	Therefore $u'\in V(R) \cup A_j$, as claimed. It follows that $R\supseteq P'$ (if $u'$ is an ancestor of $x_0$) or $P'\supseteq R$ (if $x_0$ is an ancestor of $u'$).  In either case, $R\cup P' = P_T(x)$ for some vertex $x\in\{x_0,u'\}$.

	We now adjust the ribs in the $(G,S'_j,T)$-path $J'$ to make them suitable for inclusion in our $(G,S,T)$-path $J$. Let $t \in \{1,\dots,\ell-a_i\}$.
	If $Q'_t$ is disjoint from $R$, then let $Q_t := Q'_t$ and let $u_t$ be the first vertex of $Q_t$.  Otherwise, let $u_t$ be the last vertex of $Q'_t$ in $V(R)$.
	We define $Q_t$ to be the subpath of $Q'_t$ from $u_t$ to the last vertex of $Q'_t$.  Since $V(R)$ separates $A_j$ from $V(G)\setminus A_j$ in $G$, it follows that $V(Q_t-u_t)\subseteq A_j$.

	The paths $Q_1,\ldots,Q_{\ell-a_i}$ will be ribs in our $(G,S,T)$-path $J$.  The remaining $a_i$ ribs of $J$ will be $Q_{i,1},\ldots,Q_{i,a_i}$.  Thus, $x_{i, 1},\ldots,x_{i, a_i},u_1,\ldots,u_{\ell-a_i}$ will be the attachments of $J$.   Since $\{x_{i, 1},\ldots,x_{i, a_i},u_1,\ldots,u_{\ell-a_i}\}\subseteq V(R)\cup V(P')$ and $R\cup P'= P_T(x)$ for some $x\in V(T)$, there exists a vertex $u\in \{x_{i, 1},\ldots,x_{i, a_i},u_1,\ldots,u_{\ell-a_i}\}$ such that $\{x_{i, 1},\ldots,x_{i, a_i},u_1,\ldots,u_{\ell-a_i}\}\subseteq V(P_T(u))$.  Let $P:=P_T(u)$ and   let
	\[
		J:=(P,\{Q_{i,1},\ldots,Q_{i,a_i},Q_1,\ldots,Q_{\ell-a_i}\}) \enspace .
	\]

	\begin{clm}\label{clm:gst}
		$J$ is a $(G,S,T)$-path of order $\ell$.
	\end{clm}

	\begin{clmproof}
		Clearly $J$ satisfies \ref{vertical_spine}, since $P=P_T(u)$.  It is also clear that $J$ satisfies \ref{coccyx_attachment} because $u$ is the first vertex of some rib by definition.  To show that $J$ satisfies \ref{disjoint_ribs} we must show that
		\begin{enumerate*}[label=(\arabic*)]
			\item\label{item:vertex-disjoint} the ribs of $J$ are pairwise vertex-disjoint; and
			\item\label{item:p-s-path} each rib of $J$ is a $V(P)$--$S$ path in $G$.
		\end{enumerate*}
		To show \ref{item:vertex-disjoint}, first recall that $Q'_1,\dots,Q'_{\ell-a_i}$---and therefore $Q_1,\ldots,Q_{\ell-a_i}$---are pairwise vertex disjoint and that $Q_{i,1},\dots,Q_{i,a_i}$ are pairwise vertex disjoint.
		Now, let $s \in \{1,\dots,a_i\}$ and $t \in \{1,\dots,\ell-a_i\}$.
		The paths $Q_{i,s}$ and $Q_{t}$ are vertex disjoint since $V(Q_t) \subseteq (V(R) \cup A_j) \setminus X_j'$ and $V(Q_{i,s}) \subseteq X_i \cup A_i \subseteq X_j' \cup A_i$.  Thus, the ribs of $J$ are pairwise vertex-disjoint.

		All that remains is to show \ref{item:p-s-path}; that each rib of $J$ is a $V(P)$--$S$ path in $G$.
		By construction, each single vertex rib is a vertex in $V(P)\cap S$, and for each rib with at least two vertices, the first vertex is in $V(P)\setminus S$ and the last vertex is in $S\setminus V(P)$.
		What remains is to argue that no rib has an internal vertex in $V(P)\cup S$.

		For each $s\in\{1,\ldots,a_i\}$, the rib $Q_{i,s}$ is a $(V(R)\setminus X)$--$(S\cap A_i)$ path in $G_i$.  Since $G_i$ has no vertices in $X$, $Q_{i,s}$ is a $V(R)$--$(S\cap A_i)$ path in $G$.  Therefore, $Q_{i,s}$ has no internal vertex in $V(R)$ or in $S\cap A_i$.  Since $V(G_i)\subseteq A_i\cup V(R)$, every internal vertex of $Q_{i,s}$ is in $A_i$. Therefore, $Q_{i,s}$ has no internal vertex in $S$.  Since $A_i\cap A_j=\emptyset$ and $V(P)\subseteq V(R)\cup V(P')\subseteq V(R)\cup A_j$,
		no internal vertex of $Q_{i,s}$ is in $V(P)$.  Thus, $Q_{i,s}$ has no internal vertex in $V(P)\cup S$, as required.

		Next, consider a rib $Q_t$, for some $t\in\{1,\ldots,\ell-a_i\}$.
		Then $Q_t$ is the suffix of $Q'_t$ that begins at $u_t$.
		Recall that $V(Q_t-u_t) \subset A_j$.
		Thus, $Q_t$ has no internal vertex in $V(R)$ or $S \cap A_i$.
		Since $Q'_t$ is a $V(P')$--$(S\cap A_j)$ path, $Q'_t$---and therefore $Q_t$--- has no internal vertex in $V(P')$ or $S\cap A_j$.
		Finally, since $V(P)\subseteq V(R)\cup V(P')$ and $S\cap V(C) \subset A_i \cup A_j$, $Q_t$ has no internal vertex in $V(P) \cup S$, as required.
	\end{clmproof}

	\begin{clm}\label{clm:requirements}
		$J$ satisfies \ref{item:X-spine} and \ref{item:X-degree}.
	\end{clm}

	\begin{clmproof}    
		For each $t \in \{1,\dots,\ell-a_i\}$, considering the attachment $u'_t$ of $Q'_t$ in $J'$, we see that $J$ has an attachment $u_t$ of $Q_t$ that is a descendant of $u'_t$.  This includes the vertex $u'$ that defines the spine $P'=P_T(u')$ of $J'$.  Therefore, the vertex $u$ that defines the spine $P=P_T(u)$ of $J$ is a descendant of $u'$, so $P'$ is a prefix of $P$.  Since $J'$ satisfies \ref{item:X-spine} for $X'_j$ and $X \subset X_j'$, every vertex of $X$ is in $P'$ and therefore also in $P$.  Thus, $J$ satisfies \ref{item:X-spine}.

		To see that~\ref{item:X-degree} holds, observe that each attachment $v$ of $J$ is one of the following three types:
		\begin{itemize}[nosep,nolistsep]
			\item $v=x_{i,s}$ for some $s\in\{1,\ldots,a_i\}$. By definition, $v$ is the first vertex of the $(V(R)\setminus X)$--$(S\cap A_i)$ path $Q_{i,s}$, so $v\in V(R)\setminus X$. In particular, $v\not\in X$.

			\item $v=u_t$ for some $t\in\{1,\ldots,\ell-a_i\}$ and $u_t=u'_t$ is an attachment in $J'$. By the inductive hypothesis, $u_t\not\in X_j'$ and $X \subset X_j'$. In particular, $u_t = v\not\in X$.

			\item $v=u_t$ for some $t\in\{1,\ldots,\ell-a_i\}$ and $u_t\neq u'_t$. By \ref{disjoint_ribs}, $Q'_t$ is a $V(P')$--$S'_j$ path and by the inductive hypothesis, $X_j' \subset V(P')$.
			      By definition, $u_t$ is a vertex of $Q'_t-u'_t$.  If $u_t$ is an internal vertex of $Q'_t$ then $u_t\not\in V(P')\supseteq X'_j\supseteq X$.
			      If $u_t$ is the last vertex of $Q'_t$ then $u_t\in S\subseteq V(G)\setminus X$.  In either case $v=u_t\not\in X$.
			      \qedhere
		\end{itemize}
	\end{clmproof}

	By \cref{clm:gst,clm:requirements}, $J$ is a $(G,S,T)$-path of order $\ell$ satisfying~\ref{item:first} and \ref{item:last}.  We have just shown that $\td(G-X_j',S_j')\ge 2(\ell-a_i)-1$ implies the existence of a $(G,S,T)$-path of order $\ell$ that satisfies \ref{item:first} and \ref{item:last}.

    Since the argument applies to both \((i,j)=(1,2)\) and \((i,j)=(2,1)\), we are done if either $\td(G-X'_2,S'_2)\geq 2(\ell-a_1)-1$ or $\td(G-X'_1,S'_1)\geq 2(\ell-a_2)-1$.
    It remains to consider the case in which both inequalities fail.
	We claim that this cannot happen.
	Suppose, for the sake of contradiction, that
	$\td(G - X_1',S_1') < 2(\ell-a_2)-1$ and $\td(G - X_2',S_2') < 2(\ell-a_1)-1$.  We will show that $\td(G-X,S)<2\ell-1$, thereby contradicting \ref{item:td-G-X-Y}.
	Recall that $1\leq a_1 < \ell$ and $1\leq a_2 < \ell$.
	Without loss of generality, assume that $a_1 \leq a_2$. Recall that $|X_1 \cup Y_1| \leq 2a_1$.
	Therefore,
	\[  \td(G-X,S) = \td(C,S\cap V(C)) \leq \td(C - (X_1 \cup Y_1),(S\cap V(C))\setminus(X_1\cup Y_1)) + 2a_1.  \]
	Our aim now is to show that $\td(D,S\cap V(D)) < 2(\ell-a_1)-1$ holds for every connected component $D$ of $C - (X_1 \cup Y_1)$.
	This implies that $\td(G-X,S) = \td(C,S\cap V(C)) < 2\ell-1$, which is the desired contradiction.
	Let $D$ be a connected component of $C - (X_1 \cup Y_1)$.

	If $S\cap V(D)$ is empty, then $\td(D,S\cap V(D))=0$ and we are done, so we assume that $S\cap V(D)$ is non-empty.  Recall that $V(R)$ separates $A_1$ from $A_2$ in $G$.
	Moreover, $Y_1$ separates $S \cap A_1$ from $V(R) \setminus X$ in $G_1$.
	It follows that $Y_1 \cup X$ separates $S \cap A_1$ from $S \cap A_2$ in $G$.
	In particular, $D$ intersects $S \cap A_1$ or $D$ intersects $S \cap A_2$, but not both.
	If $D$ intersects $S\cap A_1$, then $V(D) \subseteq A_1$, and since $X_1' \cap A_1 = \emptyset$, $D$ is a subgraph of $G - X_1'$.
	Additionally, $S \cap V(D) \subseteq S_1' = S \cap A_1$, and so,
	\[\td(D,S \cap V(D)) \leq \td(G - X_1',S_1') < 2(\ell-a_2)-1 \leq 2(\ell-a_1)-1.\]
	Finally, assume that $D$ intersects $S\cap A_2$.
	Since $X_2'  \subseteq X  \cup X_1 \cup Y_1$, $D$ is a subgraph of $G - X_2'$.  Moreover, $S \cap V(D) \subseteq S_2' = S \cap A_2$, and so,
	\[\td(D,S \cap V(D)) \leq \td(G - X_2',S_2') < 2(\ell-a_1)-1.\]
	Therefore, $\td(G-X,S) = \td(C,S\cap V(C)) < 2\ell-1$, contradicting~\ref{item:td-G-X-Y} and completing the proof.
\end{proof}

\section{Excluding a fan}\label{sec:ltd}

Rephrasing a statement of \citet[Lemma~17 and the proof of Theorem~2]{quickly-excluding-apex-forest}, we obtain that if $f$ is a function such that for every positive integer $\ell$, for every graph $G$, and for every set $S \subset V(G)$, the fact that $G$ has no $S$-rooted model of $P_\ell$ implies $\td(G,S) \leq f(\ell)$, then for every apex-linear forest \(H\) with at least three vertices, every \(H\)-minor-free graph \(G\) satisfies \(\ltd(G)\leq f(|V(H)|-1)\).
Therefore,~\Cref{thm:main} implies a weaker version of~\Cref{thm:ltd}, which already substantially improves~\cite[Theorem~2]{quickly-excluding-apex-forest}.
Namely, for every apex-linear forest graph $H$ with at least three vertices, every $H$-minor-free graph $G$ has $\ltd(G) \leq 2|V(H)|-4$.
By specializing the argument used in the proof of~\Cref{lem:main}, we remove the multiplicative factor of $2$ and prove~\Cref{thm:ltd}.
In fact, we prove the slightly stronger statement of~\Cref{thm:ltd-stronger}, i.e. that for every apex-linear forest $H$ with at least three vertices, for every $H$-minor-free graph $G$, and for every vertex $v\in V(G)$, there exists an elimination forest $F$ of $G$, and a layering $\mathcal{L}=(L_i \colon i\in \mathbb N)$ of $G$, such that $L_0=\{v\}$, and the width of $(F, \mathcal{L})$ is at most  $|V(H)|-2$.  This allows us to also obtain \cref{cor:td-diam}, as explained in \cref{sec:radius}.

Let $G$ be a connected graph, let $S \subseteq V(G)$, and let $T$ be a DFS tree of $G$.
A \defin{weak $(G,S,T)$-path}  of \defin{order} $\ell$ is a pair $J:=(P,\{u_1,\ldots,u_\ell\})$ that satisfies the following conditions:


\begin{enumerate}[nosep,nolistsep,label=(W\arabic*)]
	\item\label{vertical_spine_weak} $P:=P_T(u)$ for some vertex $u$ of~$G$. $P$ is called the \defin{spine} of $J$.
	\item\label{attachment_weak}
	$\{u_1,\ldots,u_\ell\}\subseteq V(P)$ and for each $s\in\{1,\ldots,\ell\}$, $u_s\in S$ or $u_s$ is adjacent, in $G$, to a connected component of $G-V(P)$ that contains a vertex of $S$.  Each vertex in $u_1,\ldots,u_\ell$ is called an \defin{attachment} of $J$.
\end{enumerate}



For every integer $\ell$ with $\ell\ge 2$, the \defin{$(\ell+1)$-vertex fan $H_\ell$} is the graph obtained by adding a dominant vertex to the $\ell$-vertex path $P_\ell$.

\begin{obs}\label{obs:fan}
	Let $G$ be a graph, let $u \in V(G)$, let $G' := G - u$ and let $S := N_{G}(u)$.
	If there exist a connected component $C'$ of $G'$ and a DFS tree $T'$ of $C'$ such that $C'$ contains a weak $(C', S\cap V(C'), T')$-path of order $\ell$, then $G$ contains $H_\ell$ as a minor.
\end{obs}
\begin{proof}
	Suppose that there exist a connected component $C'$ of $G'$ and a DFS tree $T'$ of $C'$ such that $C'$ contains a weak $(C', S\cap V(C'), T')$-path $(P,\{u_1,\dots,u_\ell\})$ of order $\ell$.
	Let $C$ be the connected component of $G - V(P)$ that contains $u$.  In the graph $G$, contract $C$ into a single vertex.
	In the resulting graph (which is a minor of $G$), $P$ is a path containing each of $u_1,\ldots,u_\ell$.
	Moreover, each of $u_1,\ldots,u_\ell$ is adjacent to the vertex resulting from the contraction of $C$.
	Clearly, this graph contains $H_\ell$ as a minor, and hence so does $G$.
\end{proof}

The proof of \cref{lem:ltd} closely parallels the proof of \cref{lem:main}. 
The main simplification here is that a weak \((G,S,T)\)-path only records its attachment vertices. Consequently, the set \(X_i\) in the proof plays the role of a separator, whereas in \cref{lem:main} we additionally needed a Menger separator \(Y_i\) of size \(a_i\). 
This is why the loss \(2a_i\) in \cref{lem:main} becomes \(a_i\) here, and ultimately why the \(2\ell-1\) lower bound is replaced by \(\ell\).

\begin{lem} \label{lem:ltd}
	Let $\ell$ be a positive integer.
	Let $G$ be a connected graph, let $T$ be a DFS tree of $G$, let $S\subseteq V(G)$, and let $X\subseteq V(G)\setminus S$ be such that
	\begin{enumerate}[nosep,nolistsep,label=\rm(\alph*),ref={Assumption (\alph*)}]
		\item $x$ is a $T$-ancestor of $v$ for each $x\in X$ and $v\in S$; \label{item:X-over-S'}
		\item $\td(G-X, S)\ge \ell$. \label{item:td-G-X-Y'}
	\end{enumerate}
	Then $G$ contains a weak $(G,S,T)$-path $J$ of order $\ell$ such that
	\begin{enumerate}[nosep,nolistsep,label=\rm(\roman*),ref={Requirement (\roman*)}]
		\item every vertex of $X$ lies in the spine of $J$; \label{item:X-spine'}\label{item:first'}
		\item no vertex of $X$ is an attachment in $J$. \label{item:X-degree'} \label{item:last'}
	\end{enumerate}
\end{lem}

\begin{proof}
	The proof is by induction on $\ell$.
	If $\ell=1$ then, since $\td(G-X,S)\ge 1$, $S$ is non-empty.  Let $u$ be a vertex of $S$ closest to the root of $T$ and let $P:=P_T(u)$. Then $J:=(P,\{u\})$ is a weak $(G,S,T)$-path of order $1$.
	By \ref{item:X-over-S'}, the spine of $J$ contains every vertex in $X$, so $J$ satisfies \ref{item:X-spine'}.
	Since $u\in S$ is the only attachment of $J$ and $X\subseteq V(G)\setminus S$, $J$ satisfies \ref{item:X-degree'}.  Thus $J$ is a weak $(G,S,T)$-path of order $1$ that satisfies \ref{item:first'} and \ref{item:last'}.  We now assume that $\ell\ge 2$.

	By the definition of $\td(\cdot,\cdot)$, there exists a connected component $C$ of $G-X$ such that $\td(C,S\cap V(C))=\td(G-X,S)\ge \ell$. Therefore, $|S\cap V(C)|\ge \ell >0$. Let $x_0$ be the lowest common ancestor of $S\cap V(C)$ in $T$.

	First, suppose that $x_0\in S$.
	Let $X' := X \cup \{x_0\}$ and $S' := (S\cap V(C)) \setminus \{x_0\}$.
	We claim that \ref{item:X-over-S'} and \ref{item:td-G-X-Y'} are satisfied for $X'$, $S'$, and $\ell' := \ell-1$.
	By the definition of $x_0$ and \ref{item:X-over-S'}, $x$ is an ancestor of $v$ for each $x \in X'$ and $v \in S'$, and by~\ref{item:td-G-X-Y'},
	\[\td(G-X',S') \ge \td(C-x_0,(S\cap V(C))\setminus\{x_0\}) \geq \td(C,S\cap V(C))-1 \geq \ell-1 = \ell'.\]
	Therefore, we can apply induction to  $X'$, $S'$, and $\ell'=\ell-1$ obtaining a weak $(G,S',T)$-path $J':=(P, \{u_1,\ldots,u_{\ell-1}\})$ in $G$ of order $\ell-1$ satisfying~\ref{item:first'} and \ref{item:last'} for $X'$.
	Since $x_0 \in X'$, $x_0$ is not an attachment in $J'$ (by \ref{item:last'}) and $x_0\in V(P)$ (by \ref{item:first'}).
	It follows that $J:=(P,\{u_1,\ldots,u_{\ell-1},x_0\})$ is a weak $(G,S,T)$-path in $G$ of order $\ell$ satisfying~\ref{item:first'} and \ref{item:last'} for $X$.


	Next, suppose that $x_0 \notin S$ and let $R := P_T(x_0)$.
	Note that $X \subseteq V(R)$ by~\ref{item:X-over-S'}.
	Let $Z$ be the set of all children $x$ of $x_0$ in $T$ such that $V(T_x)\subseteq V(C)$. Let $Z_0 \subseteq Z$ be the set of all the children $x$ of $x_0$ such that $S\cap V(C)$ intersects $V(T_x)$.  Since $x_0 \notin S$, it follows that $|Z_0| > 1$.
	Let $Z_1$ and $Z_2$ be any partition of $Z$ such that $Z_i \cap Z_0 \neq \emptyset$ for each $i \in \{1,2\}$.  For each $i\in\{1,2\}$, let $A_i := \bigcup_{x \in Z_i} V(T_{x})$, and let $G_i := G[(A_i \cup V(R)) \setminus X ]$.
	Observe that, since $T$ is a DFS tree, $V(R)$ separates $A_1$ and $A_2$ in $G$.

	Let $\{1,2\} = \{i,j\}$.
	Let $X_i$ be the set of all the vertices of $R - X$ that are adjacent in $G$ to a connected component of $G[A_i]$ that contains a vertex in $S$.
	Let $a_i:=|X_i|$.  We claim that $a_i\ge 1$. Indeed, since $A_j\cup A_i\subseteq V(C)$ and $C\subseteq G-X$, the graph $G-X$ contains an $A_j$--$(S\cap A_i)$ path $P$ and this path necessarily contains a vertex in $V(R)\setminus X$.  The last vertex of $P$ in $V(R)\setminus X$ is adjacent, in $G$, to the connected component of $G[A_i]$ that contains the vertex in $S\cap A_i\cap V(P)$.
	Note that if $a_i \geq \ell$, then $(R, X_i)$
	is a weak $(G,S,T)$-path of order $a_i \geq \ell$ that satisfies~\ref{item:X-spine'} and~\ref{item:X-degree'}, and we are done.
	Thus, we may assume that $a_i < \ell$.

	Let $X'_j := X \cup X_i$ and let $S'_j := S \cap A_j$.
	Assume that $\td(G - X'_j, S'_j) \geq \ell-a_i$.
	By the definition of $X_i$, $x$ is an ancestor of $v$ for each $x \in X'_j$ and $v \in S'_j$, so $X'_j$ and $S'_j$ satisfy \ref{item:X-over-S'}.
	Therefore, we can apply induction to $X'_j$, $S'_j$, and $\ell-a_i$, obtaining a weak $(G,S'_j,T)$-path $J':=(P',\{u'_1,\ldots,u'_{\ell-a_i}\})$ of order $\ell-a_i$ satisfying~\ref{item:first'} and \ref{item:last'} for $X'_j$.  
    Note that $X_i\subseteq V(P')$ (by \ref{item:first'}) and that $X_i$ and $\{u'_1,\ldots,u'_{\ell-a_i}\}$ are disjoint (by \ref{item:last'}).  
    Let $u$ be the deepest element among $\{u'_1,\ldots,u'_{\ell-a_i}\}\cup X_i$.
    Observe that $V(P_T(u)) \subset V(R) \cup A_j$.
    It follows that $J:=(P_T(u),\{u'_1,\ldots,u'_{\ell-a_i}\}\cup X_i)$ is a weak $(G,S,T)$-path of order $\ell$ that satisfies~\ref{item:X-spine'} and~\ref{item:X-degree'}.
	We have just shown that $\td(G-X_j',S_j')\ge \ell-a_i$ implies the existence of a weak $(G,S,T)$-path of order $\ell$.

	Now we will assume that
	$\td(G - X_1',S_1') < \ell-a_2$ and $\td(G - X_2',S_2') < \ell-a_1$ and show that this implies that $\td(G-X,S)<\ell$,
	thereby contradicting \ref{item:td-G-X-Y'}.
	Recall that $1 \leq a_1 < \ell$ and $1 \leq a_2 < \ell$.
	Without loss of generality, assume that $a_1 \leq a_2$.
	Since $|X_1| = a_1$,
	\[  \td(G-X,S) = \td(C, S\cap V(C)) \leq \td(C - X_1,(S\cap V(C))\setminus X_1) + a_1.  \]
	We want to show that $\td(D,S\cap V(D)) < \ell-a_1$ holds for every connected component $D$ of $C - X_1$, which implies $\td(C, S\cap V(C)) < \ell$, and thus $\td(G-X,S)<\ell$, as desired.
	Let $D$ be a connected component of $C - X_1$.
	If $S\cap V(D)$ is empty then $\td(D,S\cap V(D))=0$ and we are done, so we assume that $S\cap V(D)$ is non-empty.  Recall that $V(R)$ separates $A_1$ from $A_2$ in $G$.
	Moreover, $X_1$ separates $S \cap A_1$ from $V(R) \setminus X$ in $G$.
	It follows that $X_1 \cup X$ separates $S \cap A_1$ from $S \cap A_2$ in $G$.
	In particular, $D$ intersects $S \cap A_1$ or $D$ intersects $S \cap A_2$, but not both.
	If $D$ intersects $S\cap A_1$, then $V(D) \subseteq A_1$, and since $X_1' \cap A_1 = \emptyset$, $D$ is a subgraph of $G - X_1'$.
	Additionally, $V(D) \cap S \subseteq S_1' = S \cap A_1$, and so,
	\[\td(D,S \cap V(D)) \leq \td(G - X_1',S_1') < \ell-a_2 \leq \ell-a_1.\]
	Finally, assume that $D$ intersects $S\cap A_2$.
	Since $X_2' = X  \cup X_1$, $D$ is a subgraph of $G - X_2'$.
	Moreover, \(S\cap V(D)\subseteq S'_2=S\cap A_2\), and so,
	\[\td(D,S \cap V(D)) \leq  \td(G - X_2',S_2') < \ell-a_1.\]
	Therefore, $\td(G - X,S)=\td(C, S\cap V(C)) < \ell$, which is a contradiction with~\ref{item:td-G-X-Y'} that completes the proof.
\end{proof}

\Cref{lem:ltd} and~\Cref{obs:fan} imply the following.

\begin{cor}\label{cor:ltd}
	Let $G$ be a graph, let $u \in V(G)$, and let $H$ be an apex-linear forest on at least three vertices.
	If $\td(G - u, N_G(u)) \geq |V(H)|-1$, then $H$ is a minor of $G$.
\end{cor}
\begin{proof}
    Let $\ell \coloneqq |V(H)|-1$.  Assume that $\td(G - u, N_G(u))\geq \ell$.  
    There exists a connected component $C$ of $G-u$ such that $\td (C, N_G(u))\geq \ell$.  
    Let $S := N_G(u) \cap V(C)$.
    By \cref{lem:ltd} applied with $C$, any DFS tree $T$ of $C$, $S$, and $X=\emptyset$, we get that, in particular, there exists a weak $(C, S, T)$-path of order $\ell$ in $C$.  Therefore, by \cref{obs:fan}, the graph $G$ contains $H_\ell$ as a minor.  Moreover, since $H$ is an apex-linear forest with $\ell+1$ vertices, it is clear that $H$ is a minor of $H_\ell$.  This ends the proof.
\end{proof}

The following lemma is a reformulation of a result in~\cite{quickly-excluding-apex-forest}.

\begin{lem}[\cite{quickly-excluding-apex-forest}, Lemma~17] \label{lem:tdgs-to-ltd}
	Let $\ell$ be an integer with $\ell \geq 2$, let $G$ be a connected graph, and let $v\in V(G)$.
	If for every $U \subseteq V(G)$ such that $G[U]$ is connected, we have $\td(G-U, N_G(U)) < \ell$, then there exists an elimination forest $F$ of $G$ and a layering $\mathcal{L}=(L_i \colon i \in \mathbb N)$ of $G$, such that $L_0=\{v\}$, and the width of $(F, \mathcal L)$ is strictly less than $\ell$.
\end{lem}

\begin{proof}[Proof of~\Cref{thm:ltd-stronger}]
	Let $H$ be an apex-linear forest with at least three vertices and let $G$ be an $H$-minor-free graph.
	Without loss of generality, we may assume that $G$ is connected.
	Let $\ell := |V(H)| - 1$.
	We will apply~\Cref{lem:tdgs-to-ltd}.
	To check its assumption, let $U \subseteq V(G)$ be such that $G[U]$ is connected.
	Let $G'$ be obtained by contracting $U$ in $G$ to a vertex $u$.
	Note that $G'-u$ is the same graph as $G - U$, and $N_{G'}(u)=N_G(U)$.
	By~\Cref{cor:ltd} applied to $G'$ and $u$, we obtain that either $\td(G-U,N_{G}(U))=\td(G'-u,N_{G'}(u)) < \ell$ or $G'$ contains $H$ as a minor.
	Since $G'$ is a minor of $G$ and $G$ is $H$-minor-free, the latter is false.
	This indeed verifies the assumption of~\Cref{lem:tdgs-to-ltd}, and the result follows.
\end{proof}

\section{Excluding an apex-forest}\label{sec:lpw}

In this section, we prove \Cref{thm:lpw-stronger}, which implies \Cref{thm:lpw}, i.e.\ for every apex-forest $H$ with at least three vertices, and for every $H$-minor-free graph $G$, $\lpw(G) \leq |V(H)| - 2$.
Our proof is based on the material of~\cite[Section~4]{quickly-excluding-apex-forest}, which builds upon the ideas of Diestel~\cite{Diestel1995} and follows the notation of a recent paper by Seymour~\cite{seymour2023shorter}.

A \defin{separation} of a graph $G$ is a pair $(A, B)$ of subgraphs of $G$ such that $A\cup B = G$ and $E(A \cap B) = \emptyset$; the \defin{order} of $(A,B)$ is the value $|V(A)\cap V(B)|$.
For a graph \(A\) and a set \(E'\) of edges with endpoints in \(V(A)\), \defin{\(A\cup E'\)} denotes the graph with vertex set \(V(A)\) and edge set \(E(A)\cup E'\)
Let $w$ be a positive integer, and $S\subseteq V(G)$.
A separation $(A,B)$ is \defin{$(w,S)$-good} if it is of order at most $w$ and $(A,S\cap V(A))$ has a path decomposition of width at most $w-1$ whose last bag contains $V(A) \cap V(B)$ as a subset.
Note that this definition differs from the one of~\cite{quickly-excluding-apex-forest} where it was required that the path decomposition is of width at most $2w-2$.
When $(A,B)$, $(A',B')$ are separations of $G$, we write \defin{$(A,B) \leq (A',B')$}, if $A \subseteq A'$ and $B \supseteq B'$.
If moreover
the order of $(A',B')$ is at most the order of $(A,B)$,
then we say that $(A',B')$ \defin{extends} $(A,B)$.
A separation $(A,B)$ in $G$ is \defin{maximal with a property $\mathcal{P}$} if it satisfies $\mathcal{P}$ and for every separation $(A',B')$ satisfying $\mathcal{P}$ and extending $(A,B)$, we have $A'=A$ and $B'=B$. 
A separation that is maximal with the property of being $(w,S)$-good is simply called \defin{maximal $(w,S)$-good separation}.
Below, we state and prove two basic properties of maximal $(w,S)$-good separations.

\begin{obs}
	Let $w$ be a positive integer, let $G$ be a connected graph, let $S \subseteq V(G)$, and let $(A,B)$ be a separation of $G$.
	If either $(A,B)$ is maximal $(w,S)$-good or $(A,B)$ is maximal with the property of being $(w,S)$-good and satisfying $S \setminus V(A) \neq \emptyset$, then
	\begin{enumerate}[nosep,nolistsep,label=\rm(\roman*),ref={(\roman*)}]
		\item every connected component of $G-V(A)$ intersects $S$, \label{item:component-intersects-S}
		\item every vertex in $V(A)\cap V(B)$ has a neighbor in $V(B)\setminus V(A)$, \label{item:vertex-in-separator-sees-B}
		\item for every $z \in V(A)\cap V(B)$, there is a path from $z$ to a vertex of $S \setminus V(A)$ such that all internal vertices of this path are in $V(B) \setminus V(A)$.\label{item:path-z-to-S}
	\end{enumerate}
	\label{obs:finding_sep}
\end{obs}
\begin{proof}
	Note that if a separation $(A,B)$ of $G$ satisfies \ref{item:component-intersects-S} and \ref{item:vertex-in-separator-sees-B}, then $(A,B)$ satisfies also \ref{item:path-z-to-S}.
	Let $(A,B)$ be a $(w,S)$-good separation of $G$.
	We will prove that if either of \ref{item:component-intersects-S} and \ref{item:vertex-in-separator-sees-B} fails, then there is a separation $(A',B')$ of $G$ such that $(A',B')$ is $(w,S)$-good, $(A',B')$ extends $(A,B)$, $(A',B') \neq (A,B)$; and if $S \setminus V(A) \neq \emptyset$, then $S \setminus V(A') \neq \emptyset$.
	This will complete the proof contradicting the respective maximality of $(A,B)$.

	Suppose that \ref{item:component-intersects-S} fails, meaning that
	there is a connected component $C$ of $G - V(A)$ with $ S \cap V(C) = \emptyset$.
	Let $A' := G[V(A) \cup V(C)]$ and $B'$ be obtained from $B-V(C)$ by removing all edges of $A'$.
	It follows that $(A',B')$ is a separation of $G$ of the same order as $(A,B)$ and $(A,B) \leq (A',B')$.
	In particular, $(A',B')$ extends $(A,B)$ and $(A',B') \neq (A,B)$.
	Moreover, note that if $S \setminus V(A) \neq \emptyset$, then $S \setminus V(A') \neq \emptyset$ as $S \cap V(C) = \emptyset$.
	It remains to show that $(A',B')$ is $(w,S)$-good.
	Since $(A,B)$ is $(w,S)$-good, we may fix a path decomposition $\mathcal{W}$ of $(A, S \cap V(A))$ of width at most $w-1$ whose last bag contains $V(A) \cap V(B)$ as a subset.
	Note that $\mathcal{W}$ is also a path decomposition of $(A',S \cap V(A'))$.
	This is because $S \cap V(A') = S \cap V(A)$ and all neighbours of $V(C)$ in $A'$ are contained in $V(A) \cap V(B)$, which is contained in the last bag of $\mathcal{W}$.

	For \ref{item:vertex-in-separator-sees-B}, observe that if a vertex $v \in V(A) \cap V(B)$ is not adjacent to any vertex in $V(B)\setminus V(A)$, then, $(A',B') := (A \cup \{vv' \colon vv' \in E(B)\},B - v)$ is a $(w,S)$-good separation extending $(A,B)$, satisfying $(A',B') \neq (A,B)$; and if $S \setminus V(A) \neq \emptyset$, then $S \setminus V(A') \neq \emptyset$.
	This completes the proof.
\end{proof}

As mentioned, we need several statements from~\cite{quickly-excluding-apex-forest}: \Cref{lem:smaller_good_separation,lem:finding-sep,lem:pwgs-to-lpw} below.
\Cref{lem:smaller_good_separation} is almost the same as \cite[Lemma~12]{quickly-excluding-apex-forest} with a difference in the definition of $(w,S)$-good, although the proof remains the same.
\Cref{lem:finding-sep} is a slight adjustment of \cite[Lemma~14]{quickly-excluding-apex-forest}, and it can be proved along the same lines.  To obtain our version, we only need the same adaptation in the definition of $(w, S)$-good, and, in the case $t=1$ of the induction, to add the constraint that the branch set of the only vertex of the forest has to contain $v$. Besides these changes, the proof is exactly the same, but now only works when $G$ is connected.
\Cref{lem:pwgs-to-lpw} follows from the proof of \cite[Lemma~15]{quickly-excluding-apex-forest}.


\begin{lem}[\cite{quickly-excluding-apex-forest}, Lemma 12] 
	Let $G$ be a graph, let $w$ be a positive integer, let $S\subseteq V(G)$, and let $(A', B')$ and $(P, Q)$ be two separations of $G$ with $(P, Q)\leq (A', B')$.  If $(A', B')$ is $(w, S)$-good and there are $|V(P)\cap V(Q)|$ vertex-disjoint $V(P)$--$V(B')$ paths in $G$, then $(P, Q)$ is $(w, S)$-good.
	\label{lem:smaller_good_separation}
\end{lem}

\begin{lem}[\cite{quickly-excluding-apex-forest}, Lemma 14, adapted] \label{lem:finding-sep}
	Let $w$ be a positive integer, let $G$ be a connected graph, let $v\in V(G)$ and let $S \subseteq V(G)$ such that $\pw(G,S) \geq  w$.
	If $F$ is a forest on at most $w$ vertices, and $y$ is a vertex of $F$ then there is a separation $(A,B)$ of $G$ such that
	\begin{enumerate}[nosep,nolistsep,label=\rm(\roman*),ref={Requirement (\roman*)}]
		\item $|V(A) \cap V(B)| \leq |V(F)|$, \label{item:find_wTS_i}
		\item there is a $(V(A) \cap V(B))$-rooted model of $F$ in $A$ such that $v$ is in the branch set of $y$ , and \label{item:find_wTS_ii}
		\item $(A,B)$ is maximal $(w,S)$-good. \label{item:find_wTS_iii}
	\end{enumerate}
\end{lem}

\begin{lem}[\cite{quickly-excluding-apex-forest}, Lemma~15, adapted] \label{lem:pwgs-to-lpw}
	Let $w$ be a positive integer, let $G$ be a connected graph, and let $u\in V(G)$.
	If for every $U \subseteq V(G)$ such that $G[U]$ is connected with $S := N_G(U)$, we have $\pw(G-U,S) < w$, then there exists a path decomposition $\mathcal W$ of $G$ and a layering $(L_i \colon i \in \mathbb N)$ of $G$, such that $L_0=\{u\}$, and the width of $(\mathcal W, \mathcal L)$ is at most $w$.  In particular, $\lpw(G) \leq  w$.
\end{lem}

We will also use the following variant of Menger's theorem.
\begin{thm}[Menger]
	Let $G$ be a graph, and let $(A, B)$ and $(A_1, B_1)$ be separations of $G$.  If $(A, B)\leq (A_1, B_1)$, then there exists a separation $(P, Q)$ such that $(A, B)\leq (P, Q)\leq (A_1, B_1)$ and there are $|V(P)\cap V(Q)|$ pairwise vertex-disjoint $V(A)-V(B_1)$ paths in $G$.
	\label{thm:menger_separation}
\end{thm}

We are ready to state and prove the main technical lemma.
\cref{lem:finding_F+} follows the same general extension-versus-obstruction pattern as \cref{lem:main,lem:ltd}. 
The difference is that here the DFS-tree induction and focused-treedepth contradiction are replaced by maximal \((w,S)\)-good separations: Menger’s theorem either provides enough disjoint paths to move the rooted model to a larger boundary, or a smaller intermediate separation contradicts maximality.

\begin{lem}
	Let $w$ be a positive integer, let $G$ be a connected graph, let $v\in V(G)$ and let $S\subseteq V(G)$ such that $\pw(G, S)\geq w$.  Let $F$ be a forest on $w+1$ vertices, let $x$ be a vertex of degree at most $1$ of $F$.
	If $x$ is of degree $1$ in $F$, then let $y$ be its neighbour in $F$ and if $x$ is of degree $0$ in $F$, then let $y$ be any vertex of $F$ distinct from $x$.
	Then, either there is a separation $(A_1, B_1)$ of $G$ such that
	\begin{enumerate}[nosep,nolistsep,label=\rm(1.\arabic*),ref={Requirement (1.\arabic*)}]
		\item there is a $(V(A_1)\cap V(B_1))$-rooted model of $F$ in $A_1$, \label{item:1.1}
		\item $V(A_1)\ne V(G)$, and every connected component of $G-V(A_1)$ intersects $S$,\label{item:1.2}
		\item every vertex of $V(A_1)\cap V(B_1)$ has a neighbour in $V(B_1)\setminus V(A_1)$; \label{item:1.3}
	\end{enumerate}
	or there is a separation $(A_2, B_2)$ of $G$ such that
	\begin{enumerate}[nosep,nolistsep,label=\rm(2.\arabic*),ref={Requirement (2.\arabic*)}]
		\item there is a $(V(A_2)\cap V(B_2))$-rooted model of $F-x$ in $A_2$ such that $v$ is in the branch set of~$y$,\label{item:2.1}
		\item there exists a vertex $s\in S\setminus V(A_2)$ such that for every $z\in V(A_2)\cap V(B_2)$, there exists a path between $s$ and $z$ all of whose internal vertices are in $V(B_2)\setminus V(A_2)$. \label{item:2.2}
	\end{enumerate}
	\label{lem:finding_F+}
\end{lem}

\begin{proof}
	By~\Cref{lem:finding-sep} applied to $w := |V(F)|-1$, $G$, $v$, $S$,  $F - x$, and $y$, we obtain a separation $(A_0, B_0)$ of $G$ satisfying
	\begin{enumerate}[nosep,nolistsep,label=\rm(0.\arabic*),ref={Property (0.\arabic*)}]
		\item $|V(A_0) \cap V(B_0)| \leq |V(F-x)| = w$, \label{item:find_wTS_i-0}
		\item there is a $(V(A_0) \cap V(B_0))$-rooted model $\{M_z : z \in V(F-x)\}$ of $F-x$ in $A_0$ such that $v \in V(M_y)$, and \label{item:find_wTS_ii-0}
		\item $(A_0,B_0)$ is maximal $(w,S)$-good. \label{item:find_wTS_iii-0}
	\end{enumerate}
	By~\ref{item:find_wTS_iii-0}, we may fix a path decomposition $(W_1, \ldots , W_m)$  of $(A_0, S\cap V(A_0))$ of width at most $w-1$ such that $V(A_0) \cap V(B_0) \subset W_m$.
	Let $u$ be a vertex of $V(A_0)\cap V(B_0) \cap V(M_y)$.
	By~\Cref{obs:finding_sep}.\ref{item:vertex-in-separator-sees-B}, there exists $u' \in V(B_0) \setminus V(A_0)$ such that $u$ and $u'$ are adjacent in $G$.
	Let $A:=G[V(A_0)\cup \{u'\}]$ and let $B$ be obtained from $B_0$ by removing all edges of $A$.
	In particular, $(A,B)$ is a separation of $G$.
	Letting $M_x := G[\{u'\}]$, by~\ref{item:find_wTS_ii-0}, we obtain that $\{M_z : z \in V(F)\}$ is a $(V(A)\cap V(B))$-rooted model of $F$ in $A$.
	Since $V(A_0)\cap V(B_0)\subseteq W_m$, and the neighbourhood of $u'$ in $A$ is contained in $V(A_0)\cap V(B_0)$, $(W_1, \ldots , W_m, W_m \cup\{u'\})$ is a path decomposition of $(A, S\cap V(A))$.
	Moreover, since $|W_m \cup\{u'\}|=|W_m|+1\leq w+1$,  this path decomposition has width at most $w$.
	We also have $V(A)\cap V(B) = (V(A_0)\cap V(B_0))\cup \{u'\}\subseteq W_m \cup \{u'\}$.  Therefore, $(A, B)$ is $(w+1, S)$-good.

	The path decomposition $(W_1,\dots,W_m)$ witnesses that $\pw(A_0,S \cap V(A_0)) < w$ and by assumptions, $\pw(G,S) \geq w$, hence, $V(A_0) \neq V(G)$.
	Additionally, by~\cref{obs:finding_sep}.\ref{item:component-intersects-S}, every connected component of $G - V(A_0)$ intersects $S$, thus, $S\setminus V(A_0)\ne \emptyset$.
	It follows that either $S\setminus V(A) \ne \emptyset$ or $S\setminus V(A_0)= \{u'\}$.

	First, suppose that $S\setminus V(A_0)= \{u'\}$.
	We claim that \ref{item:2.1} and \ref{item:2.2} hold for $(A_2,B_2) := (A_0,B_0)$.
	Indeed, \ref{item:2.1} follows from \ref{item:find_wTS_ii-0}, and \ref{item:2.2} follows from~\Cref{obs:finding_sep}.\ref{item:path-z-to-S} as $|S\setminus V(A_0)| = 1$.

	It remains to consider the case where $S\setminus V(A)\ne \emptyset$.
	Recall that $(A, B)$ is $(w+1, S)$-good and let $(A_1, B_1)$ be a maximal separation with the property of being $(w+1, S)$-good and satisfying $S\setminus V(A_1) \ne \emptyset$ and $(A, B)\leq (A_1, B_1)$.
	We claim that \ref{item:1.1}, \ref{item:1.2}, and \ref{item:1.3} hold for $(A_1,B_1)$.
	Now, \ref{item:1.2} and \ref{item:1.3} follow from \Cref{obs:finding_sep}.\ref{item:component-intersects-S} and~\ref{item:vertex-in-separator-sees-B}.

	Finally, we argue that $(A_1,B_1)$ satisfies~\ref{item:1.1}, that is, there is a $(V(A_1)\cap V(B_1))$-rooted model of $F$ in $A_1$.
	By \cref{thm:menger_separation}, there exists a separation $(P, Q)$ such that $(A, B)\leq (P, Q)\leq (A_1, B_1)$ and a family $\mathcal L$ of $|V(P)\cap V(Q)|$ vertex-disjoint $V(A)$--$V(B_1)$ paths in $G$.  Observe that each path in $\mathcal L$ is also a $(V(A)\cap V(B))$--$(V(A_1)\cap V(B_1))$ path.
	If $|V(P)\cap V(Q)|\leq w$, then, by~\cref{lem:smaller_good_separation} applied to $(A_1, B_1)$ and $(P, Q)$, the separation $(P, Q)$ is $(w, S)$-good.
	In this case, as $(A_0,B_0) \neq (A,B)$ and $(A_0,B_0) \leq (A,B) \leq (P,Q)$, we obtain a contradiction with $(A_0, B_0)$ being maximal $(w, S)$-good.
	Thus, $|V(P)\cap V(Q)|\geq  w+1 = |V(F)|$.
	Using the paths in $\mathcal{L}$, we can extend a $(V(A)\cap V(B))$-rooted model of $F$ in $A$ to a $(V(A_1)\cap V(B_1))$-rooted model of $F$ in $A_1$, which completes the proof.
\end{proof}

\begin{cor}
	Let $G$ be a graph, let $u\in V(G)$, and let $H$ be an apex-forest on at least three vertices.
	If $\pw(G-u, N_G(u))\geq |V(H)|-2$, then, $H$ is a minor of $G$.
	\label{cor:finding_F+}
\end{cor}
\begin{proof}
	Let $h \in V(H)$ be such that $F := H - h$ is a forest.
	Let $x$ be a vertex of degree at most $1$ in $F$.
	If $x$ is of degree $1$ in $F$, then let $y$ be its neighbour in $F$ and if $x$ is of degree $0$ in $F$, then let $y$ be any vertex of $F$ distinct from $x$.
	Note that $F$ has at least two vertices, hence, $y$ is well-defined.
	Let $C$ be a connected component of $G - u$ such that $\pw(G-u, N_G(u)) = \pw(C, N_G(u) \cap V(C))$.
	Let $S := N_G(u) \cap V(C)$, $v \in S$, and $w := |V(H)| - 2$.
	We apply \cref{lem:finding_F+} to $w$, $C$, $v$, $S$, $F$, $x$, and $y$.
	There are two possible outcomes.

	First, suppose that the outcome of \cref{lem:finding_F+} is a separation $(A_1,B_1)$ of $C$ satisfying \ref{item:1.1}, \ref{item:1.2}, and \ref{item:1.3}.
	Let $\{M_z : z \in V(F)\}$ be a $(V(A_1)\cap V(B_1))$-rooted model of $F$ in $A_1$ whose existence is assured by \ref{item:1.1}.
	Let $M_h$ be the connected component of $G - V(A_1)$ containing $u$.
	By~\ref{item:1.2}, $C - V(A_1)$ has at least one connected component and each of them intersects $S \subset N_G(u)$.
	It follows that $V(C - V(A_1)) \subset V(M_h)$.
	Moreover by~\ref{item:1.3}, $M_h$ is adjacent to $M_z$ for every $z \in V(F)$, hence, $\{M_z : z \in V(H)\}$ is a model of $H$ in $G$, as desired.

	Second, suppose that the outcome of \cref{lem:finding_F+} is a separation $(A_2,B_2)$ of $C$ satisfying \ref{item:2.1} and \ref{item:2.2}.
	Let $\{M_z : z \in V(F-x)\}$ be a $(V(A_2)\cap V(B_2))$-rooted model of $F-x$ in $A_2$ whose existence is assured by \ref{item:2.1}.
	Let $s \in S \setminus V(A_2)$ be as in \ref{item:2.2}, and, for every $z\in V(A_2)\cap V(B_2)$, let $P_{z}$ be a path from $s$ to $z$, whose internal vertices all belong to $V(B_2)\setminus V(A_2)$.
	We define $M_x := G[\{u\}]$ and $M_h := G[\bigcup_{z\in V(A_2)\cap V(B_2))}V(P_z) \setminus\{z\}]$.  By construction, they are both connected, and are disjoint.
	Since $v \in V(M_y)$ and $v \in S \subset N_G(u)$, $M_y$ and $M_x$ are adjacent in $G$.
	By construction, $M_h$ is adjacent to $M_z$ for every $z \in V(F)$, hence, $\{M_z : z \in V(H)\}$ is a model of $H$ in $G$, as desired.
\end{proof}

\begin{proof}[Proof of~\Cref{thm:lpw-stronger}]
	Let $H$ be an apex-forest with at least three vertices and let $G$ be an $H$-minor-free graph.
	Without loss of generality, we may assume that $G$ is connected.
	Let $w := |V(H)| - 2$.
	We will apply~\Cref{lem:pwgs-to-lpw}.
	To check its assumption, let $U \subseteq V(G)$ be such that $G[U]$ is connected, and let $S := N_G(U)$.
	Let $G'$ be obtained by contracting $U$ in $G$ to a vertex $u$.
	Note that $G'-u$ is the same graph as $G - U$.
	By~\Cref{cor:finding_F+} applied to $G'$, $u$, and $H$ we obtain that either $\pw(G'-u,S) < w$ or $G$ contains $H$ as a minor.
	Since $G'$ is a minor of $G$ and $G$ is $H$-minor-free, the latter is false.
	This indeed verifies the assumption of~\Cref{lem:pwgs-to-lpw}, and the result follows.
\end{proof}

\section{Lower bounds}
\label{sec:lower-bounds}

In this section, we give lower bounds witnessing the tightness of \cref{thm:lpw}, \cref{thm:ltd} and \cref{cor:td-diam}.
All these are witnessed by the same infinite family of graphs.

Recall that a \defin{block} in a graph is a subgraph that is $2$-connected or isomorphic to $K_1$ or $K_2$, and inclusion-wise maximal with this property.

For all integers $t$ and $r$ with $t \geq 2$ and $r \geq 1$, we construct graphs \defin{$G_{t,r}$}.
The construction is recursive on $r$.
For the base case, let $G_{t,1}$ be a complete graph on $t$ vertices.
We say that all the vertices of $G_{t,1}$ are on \defin{level~$1$}. 
Next, having constructed $G_{t,r-1}$, let $G_{t,r}$ be obtained from a copy $J$ of $G_{t,r-1}$ by adding for each vertex $v \in V(J)$, three disjoint copies $J^1_v$, $J^2_v$ and $J^3_v$ of a complete graph on $t-1$ vertices and adding all edges between $v$ and vertices of $J^i_v$ for every $i\in \{1, 2, 3\}$.  We will identify the vertices of $J$ with those of $G_{t, r-1}$, which means that $V(G_{t, r})=V(G_{t, r-1})\cup (\bigcup_{v\in V(J), i\in \{1, 2, 3\}} V(J^i_v))$.
We say that the vertices in $\bigcup_{v \in V(J), i\in \{1, 2, 3\}}V(J^i_v)$ are on \defin{level~$r$}. 
For every $r' \in \{1,\dots,r-1\}$, we say that the vertices on level $r'$ in $G_{t, r-1}$ are also on level $r'$ in $G_{t, r}$.
Note that every block of $G_{t,r}$ is a complete graph on $t$ vertices,  and either is $G_{t,1}$ or is such that all of its vertices except one are on the same level, say $s$, and the other vertex $v$ of the block is on a level lower than $s$.  In the latter case, we set $v$ to be the \defin{root} of this block. 
We additionally assign an arbitrary vertex to be the \defin{root} of $G_{t,1}$.
Therefore, we obtain the notion of the \defin{root of a block} in $G_{t,r}$.

\begin{figure}
    \centering
    \includegraphics[width=0.75\linewidth]{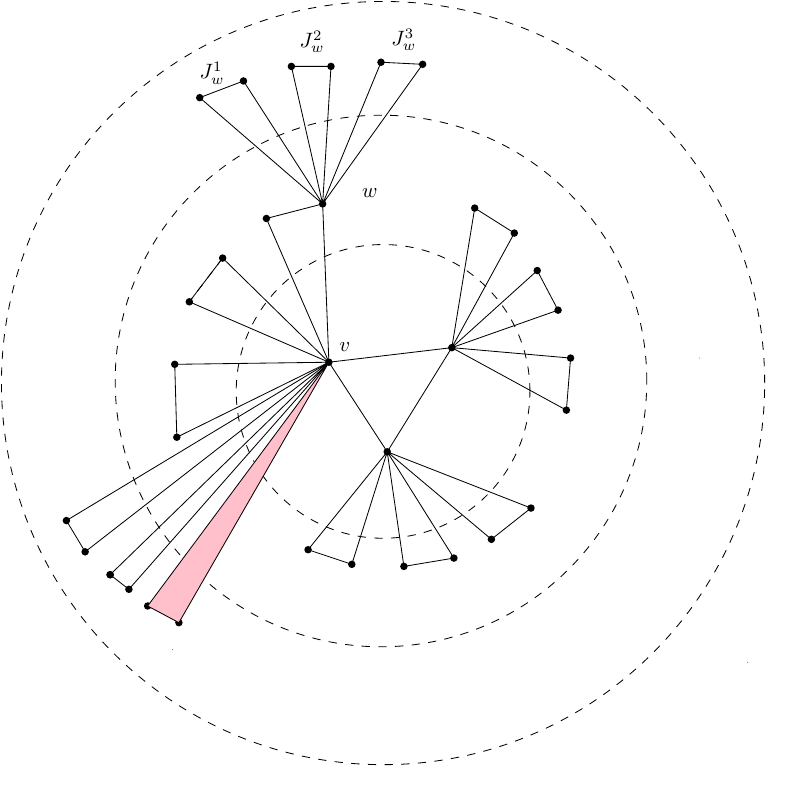}
    \caption{Partial drawing of $G_{3, 3}$.  The concentric dashed circles represent the levels.  The third level is not entirely drawn.  The pink triangle highlights a block of $G_{3, 3}$ with root $v$.  Observe that this block is indeed a complete graph with $t=3$ vertices.}
    \label{fig:placeholder}
\end{figure}

Let $G$ be a graph with an elimination forest $F$.
Let $u \in V(G)$.
When $u$ is a root or a leaf of $F$, let $F' := F-u$.
Otherwise, let $F'$ be obtained from $F-u$ by adding edges between the parent of $u$ in $F$ and all children of $u$ in $F$.
Finally, all roots of $F$ distinct from $u$ are assigned to be roots of $F'$ and if $u$ is a root of $F$, then all children of $u$ in $F$ become roots of $F'$.
Note that $F'$ is an elimination forest of $G - u$.
In particular, when $J$ is an induced subgraph of $G$, we obtain an elimination forest $F'$ of $J$ from $F$ by repeatedly removing vertices of $V(G) \setminus V(J)$ from $G$ and performing the above operation on the elimination forest.
In this case, we will call $F'$, the \defin{elimination forest of $J$ induced by $F$}.

\begin{lem}\label{lem:lb}
	Let $t$ and $r$ integers with $t \geq 2$ and $r \geq 1$.
	\begin{enumerate}[nosep,nolistsep,label=\arabic*.,ref={Property \arabic*}]
		\item Every block of $G_{t,r}$ is a complete graph on $t$ vertices. \label{lem:lb-blocks}
		\item $\radius(G_{t,r}) \leq r$. \label{lem:lb-radius}
		\item $\td(G_{t,r}) \geq 1 + r(t-1)$. \label{lem:lb-td}
	\end{enumerate}
\end{lem}

\begin{proof}
	We have already seen that \ref{lem:lb-blocks} is clear from the construction.

	In order to prove~\ref{lem:lb-radius}, we prove by induction on $r$ that every vertex on level $1$ in $G_{t,r}$ is in distance at most $r$ from every vertex of $G_{t,r}$.
	This clearly implies that $\radius(G_{t,r}) \leq r$.
	For $r=1$ the assertion is clear as $G_{t,1}$ is a complete graph on $t$ vertices.
	Assume that $r > 1$ and let $v$ be a vertex of $G_{t,r}$ on level $1$.
	By induction, $v$ is in distance at most $r-1$ from every vertex of $G_{t,r}$ on levels at most $r-1$.
	Let $u$ be a vertex of $G_{t,r}$ on level $r$.
	Then, by construction, $u$ has a neighbor $u'$ in $G_{t,r}$ on level at most $r-1$, thus, $u$ is in distance at most $r$ from $v$ in $G_{t,r}$.
	This shows~\ref{lem:lb-radius}.

	We prove \ref{lem:lb-td} by induction on $r$.
	For $r=1$, the assertion is clear as $G_{t,1}$ is a complete graph on $t$ vertices, so $\td(G_{t,1}) = t$.
	Assume that $r > 1$.
	Let $F$ be an elimination forest of $G_{t,r}$ such that the vertex-height of $F$ is equal to $\td(G_{t,r})$.
	Let $J$ and $J^i_v$ for $v\in V(J)$ and $i\in \{1, 2, 3\}$ be as in the construction of $G_{t,r}$.
	In particular, $J$ is a copy of $G_{t,r-1}$ and $J^i_v$ is a copy of a complete graph on $t-1$ vertices for each $v \in V(J)$ and $i\in \{1, 2, 3\}$.
	Let $X := \bigcup_{v \in V(J), i\in \{1, 2, 3\}} V(J^i_v)$.
	It follows that the graph $G_{t,r-1}$ is isomorphic to $G_{t,r} - X$.
	Let $F'$ be the elimination forest of $G_{t,r} - X$ induced by $F$.
	Let $v$ be a leaf of $F'$.
	Since $V(J^1_v) \cup \{v\}$ is a clique in $G_{t,r}$, the vertices $V(J^1_v) \cup \{v\}$ are in one root-to-leaf $R$ path in $F$.
	Also, by definition, every vertex of $P_{F'}(v)$ lie in $R$ in $F$.
	It follows that $|V(R)| \geq |P_{F'}(v)| + |V(J^1_v)| = |P_{F'}(v)| + (t-1)$.
	Taking a leaf $v$ of $F'$ that maximizes $|P_{F'}(v)|$, we obtain that the vertex-height of $F$ is at least the vertex-height of $F'$ plus $t-1$.
	By induction, $\td(G_{t,r-1}) \geq 1 + (r-1)(t-1)$, thus, the vertex-height of $F'$ is at least $1 + (r-1)(t-1)$ and finally the vertex-height of $F$ (equal to $\td(G_{t,r})$) is at least $1 + (r-1)(t-1) + (t-1) = 1 + r(t-1)$, as desired.
	This completes the proof of~\ref{lem:lb-td}.
\end{proof}

\cref{cor:td-diam} states that for every apex-linear forest $H$ with at least two vertices, and for every connected graph $G$, if $G$ is $H$-minor-free, then $\td(G) \leq (|V(H)|-2)\cdot \radius(G)+1$.  The fact that this is tight is a direct corollary of \cref{lem:lb}.
Recall that for every integer $t$ with $t\ge 2$, the $(t+1)$-vertex fan $H_t$ is the graph obtained by adding a dominant vertex to the $t$-vertex path $P_t$.
Note that $H_t$ is an apex-linear forest, and so, also an apex-forest.
Moreover, $H_t$ is $2$-connected.
\begin{cor}
	For all integers $t$ and $r$ integers with $t \geq 2$ and $r \geq 1$, there exists a graph $G$ of radius at most $r$ that is $H_t$-minor-free and such that $\td(G) \geq (|V(H_t)|-2)\cdot \radius(G) + 1$
\end{cor}
\begin{proof}
	Take $G=G_{t,r}$.
	By~\ref{lem:lb-blocks} in~\Cref{lem:lb}, as $H_t$ is $2$-connected, $G$ is $H_t$-minor-free.
	By~\ref{lem:lb-radius}, the radius of $G$ is at most $r$ and by ~\ref{lem:lb-td}, we have $\td(G) \geq (t-1)r + 1 \geq (|V(H_t)|-2)\cdot \radius(G) + 1$.
\end{proof}

For all positive integers \(r\) and \(d\), let \defin{\(T_{r,d}\)} denote the complete rooted \(d\)-ary tree of vertex-height \(r\).
Tightness of~\cref{thm:lpw,thm:ltd} is witnessed by $G_{t,t'}$ where $t'$ depends on $t$.
The next lemma abstracts a property that we need to show tightness of both theorems.

\begin{lem}
	Let $t$ and $r$ be integers with $t \geq 2$ and $r \geq 1$.
	Let $\mathcal L$ be a layering of $G_{t, r}$ satisfying
	\begin{enumerate}[nosep,nolistsep,label=\rm(\alph*),ref={Assumption (\alph*)}]
		\item for every block $B$ of $G_{t,r}$, we have $|V(B) \cap L_B|\geq 2$, where $L_B$ is the layer of $\mathcal{L}$ containing the root of $B$. \label{item:blocks-layering}
	\end{enumerate}
	Then, there is a set $X_r\subset V(G_{t, r})$ such that
	\begin{enumerate}[nosep,nolistsep,label=(\roman*),ref=(\roman*)]
		\item $X_r$ is contained in a single layer of $\mathcal{L}$, \label{item:single-layer}
		\item $G_{t,r}[X_r]$ is isomorphic to $T_{r,3}$, and \label{item:Tr3}
		\item in this isomorphism the leaves of $T_{r,3}$ correspond to vertices of $G_{t,r}$ on level~$r$. \label{item:leaves}
	\end{enumerate}
	\label{lem:lb_layer}
\end{lem}

\begin{proof}
	The proof is by induction on $r$.  For the base case of $r=1$, we choose $X_1$ to be the one-vertex set containing the root of the unique block of $G_{t, 1}$. It is straightforward to check that $X_1$ verifies the assertion of the lemma.

	For the inductive step, let $r$ be a positive integer.
	Let $\mathcal L=\{L_i \ | \ i \in \mathbb N\}$ be a layering of $G_{t, r+1}$ satisfying \ref{item:blocks-layering} and let $\mathcal L':= \{L_i\cap V(G_{t, r}) \ | \ i\in \mathbb N\}$.
	Thus, $\mathcal L'$ is a layering of $G_{t, r}$ satisfying \ref{item:blocks-layering}.
	By induction, there exists a set $X_r$ respecting the assertion of the lemma for $r$ and $\mathcal{L}'$.
	Since $X_r$ is contained in a single layer of $\mathcal{L}'$, it is contained in a single layer of $\mathcal{L}$, say $L \in \mathcal{L}$.
	Let $X'_r$ be the subset of $X_r$ corresponding to the leaves of $T_{r, 3}$ (as in \ref{item:leaves}).

	Let $v \in X_r'$, $i \in \{1,2,3\}$, and let $J^i_v$ be as in the construction of $G_{t, r+1}$.
	The subgraph $B := G_{t,r+1}[\{v\} \cup V(J^i_v)]$ is a block of $G_{t,r+1}$ with the root $v$ and all other vertices on level $r+1$.
	By~\ref{item:blocks-layering}, $|V(B) \cap L| \geq 2$, hence, let $w_{i,v}$ be a vertex in this intersection distinct from $v$.

	Note that for all $v,v' \in X_r'$ and $i,i' \in \{1,2,3\}$, if $(v,i) \neq (v',i')$, then $w_{i,v} \neq w_{i',v'}$.
	We define $Y_r := \{w_{i,v} :  v\in X'_r, \ i\in \{1, 2, 3\} \}$, and $X_{r+1}:= X_r\cup Y_r$.
	It suffices to check that $X_{r+1}$ satisfies the requirements.

	Since $X_r \subset L$ and $Y_r \subset L$, we have $X_{r+1} \subset L$, and~\ref{item:single-layer} holds.
	We extend the isomorphism between $G_{t,r}[X_r]$ and $T_{r,3}$ to an isomorphism between $G_{t, r+1}[X_{r+1}]$ and $T_{r+1,3}$ as follows.
	Let $v \in X_r$ correspond to a leaf $x$ of $T_{r,3}$.
	Then, we assign the three leaves adjacent to $x$ in $T_{r+1,3}$ with the vertices $w_{1,v}, w_{2,v}$ and $w_{3,v}$.
	It clearly follows that in this isomorphism the leaves of $T_{r+1,3}$ correspond to vertices of $Y_r$, that is, the vertices $G_{t,r+1}$ on level~$r+1$.
	We obtained \ref{item:Tr3} and~\ref{item:leaves}, thus, the proof is complete.
\end{proof}

The following corollary shows that \cref{thm:lpw} is tight, and also, since $\lpw(G)\leq \ltd(G)$ for every graph $G$, that \cref{thm:ltd} is tight as well.
\begin{cor}
	Let $t$ be an integer with $t \geq 4$.
	There exists an $H_{t-1}$-minor-free $G$ such that $\lpw(G)\geq t - 2$.
\end{cor}

\begin{proof}
	Let $G := G_{t-1, t-1}$.
	By~\ref{lem:lb-blocks} in~\Cref{lem:lb}, as $H_{t-1}$ is $2$-connected, $G$ is $H_{t-1}$-minor-free.
	It remains to show that $\lpw(G)\geq t-2$.

	Let $\mathcal L$ be a layering of $G$.  First, assume that for every block $B$ of $G$, there is another vertex of $V(B)$ that is in the same layer as the root of $B$.
	Then, by \cref{lem:lb_layer}, there exists a layer $L\in \mathcal L$ and a subset of vertices $X\subset L$ such that $G[X]$ contains $T_{t-1, 3}$  as a subgraph.
	Since $\pw(T_{t-1, 3})= t-2$, we have $\lpw(G)\geq t-2$, as desired.
	Thus, we can assume that there exists a block $B$ of $G$ such that all vertices of $B$ except its root are in the same layer of $\mathcal L$.
	However, since every block of $G$ is a complete graph on $t-1$ vertices, this means that there exists a clique $K$ of size at least $t-2$ in $G$ whose vertices are all in the same layer of $\mathcal L$.
	It follows that in any path decomposition of $G$ there is a bag containing all vertices of $K$, hence, again $\lpw(G)\geq t -2$, as desired.
\end{proof}

\bibliographystyle{abbrvnat}
\bibliography{bibliography}
\end{document}